\documentclass[11pt,reqno]{article}

\usepackage[top=2.5cm, bottom=2.5cm, left=2.5cm, right=2.5cm]{geometry}	%Margins
\usepackage{graphicx}        % standard LaTeX graphics tool

\usepackage{color} 
\usepackage{amsthm}
\usepackage{amsmath}
\usepackage{amssymb}
\usepackage{mathrsfs}
\usepackage{mdwlist} %No spacing in lists
\usepackage{bm} % bold math brackets
\numberwithin{equation}{section}
\usepackage[autostyle]{csquotes}
\usepackage{hyperref}
\usepackage{color}
\usepackage{enumitem}
\usepackage{chngcntr}
\usepackage{float} %Table placement issues
\usepackage{titlesec}
\usepackage{soul}
\usepackage{enumitem}
\usepackage{multirow}

\newtheorem{theorem}{Theorem}
\newtheorem{lemma}[theorem]{Lemma}
\newtheorem{example}[theorem]{Example}

\newtheorem{condition}[theorem]{Condition}

\numberwithin{theorem}{section}
\renewcommand{\cal}[1]{\mathcal{#1}}

\renewcommand{\r}{\mathbb{R}}
\newcommand{\rn}{\mathbb{R}^n}

\newcommand{\rnn}{\mathbb{R}^{n\times n}}
\newcommand{\rnm}{\mathbb{R}^{n\times m}}

\newcommand{\n}{\mathbb{N}}

\newcommand{\z}{\mathbb{Z}}

\newcommand{\nn}{\mathbb{N}^n}

\newcommand{\norm}[1]{\left|\left|{#1}\right|\right|}

\newcommand{\Pb}{\mathbb{P}}
\newcommand{\Eb}{\mathbb{E}}
\newcommand{\Ebb}[1]{\mathbb{E}\left[#1\right]}

\newcommand{\Pbb}[1]{\mathbb{P}\left(#1\right)}

\newcommand{\pol}{\r[x]}

\newcommand{\nnd}{\nn_d}

\newcommand{\rrd}{\r^{r(d)}}
\newcommand{\rrdrd}{\r^{r(d)\times r(d)}}

\newcommand{\iprod}[2]{\left\langle #1,#2\right\rangle}
\newcommand{\mmag}[1]{\left|#1\right|}

\titleformat{\subsection}[runin]
  {\normalfont\normalsize\bfseries}{\thesubsection}{1em}{}
  
  \titleformat{\section}[runin]
  {\normalfont\large\bfseries}{\thesection}{1em}{}
  
\title{Bounding stationary averages of polynomial diffusions via semidefinite programming}

\author{
  Juan Kuntz\footnote{Department of Mathematics and Department of Bioengineering, Imperial College London, London, SW7 2AZ,UK (corresponding author: jk208@ic.ac.uk).}
  \and
  Michela Ottobre\footnote{Mathematics Department, Heriot Watt University, Edinburgh, EH14 4AS, UK m.ottobre@hw.ac.uk).}
  \and
  Guy-Bart Stan\footnote{Department of Bioengineering, Imperial College London, London, SW7 2AZ, UK     (g.stan@ic.ac.uk).}
  \and
  Mauricio Barahona\footnote{Department of Mathematics, Imperial College London, London, SW7 2AZ, UK  (corresponding author: m.barahona@ic.ac.uk).}
}

\begin{document}

\maketitle
% REQUIRED
\begin{abstract}
We introduce an algorithm based on semidefinite programming that yields increasing (resp. decreasing) sequences of lower (resp. upper) bounds on polynomial stationary averages of diffusions with polynomial drift vector and diffusion coefficients. The bounds are obtained by optimising an objective, determined by the stationary average of interest, over the set of real vectors defined by certain linear equalities and semidefinite inequalities which are satisfied by the moments of any stationary measure of the diffusion. We exemplify the use of the approach through several applications: a Bayesian inference problem; the computation of Lyapunov exponents of linear ordinary differential equations perturbed by multiplicative white noise; and a reliability problem from structural mechanics. Additionally, we prove that the bounds converge to the infimum and supremum of the set of stationary averages for certain SDEs associated with the computation of the Lyapunov exponents, and we provide numerical evidence of convergence in more general settings.
\end{abstract}

\section{Introduction.}%MICHELA MOLECULAR BOOK, EXAMPLES FOR POTENTIALS, EG LENNARD-JONES
Stochastic differential equations (SDEs) and the diffusion processes they generate are important modelling tools in numerous scientific fields, such as chemistry, economics, physics, biology, finance, and epidemiology \cite{Kloeden2011}. Stationary measures are to SDEs what fixed points are to deterministic systems: if the SDE is stable, then its stationary measures determine its long-term statistics. More concretely, both the ensemble averages and the time averages of the process converge to averages with respect to a stationary measure (which we call stationary averages). 
For the majority of SDEs, there are no known analytical expressions for their stationary measures. Consequently, large efforts have been directed at developing computational tools that estimate stationary averages and, more generally, at developing tools that can be used to study the long-term behaviour of SDEs. Among these, most prominent are Monte Carlo discretisation schemes, PDE methods  (finite-difference, finite-element,  Galerkin methods),  path integral methods, moment closure methods, and linearisation techniques (see, e.g., \cite{Schueller1997,Dunne1997,Talay1990,Talay2002}). 

The purpose of this paper is to introduce a new algorithm that provides an alternative approach to the analysis of stationary measures. It uses semidefinite programming to compute hard bounds on polynomial stationary averages of polynomial diffusions. Our approach has distinct advantages: (i) it returns monotone sequences of both upper and lower bounds on the stationary averages, hence quantifying precisely the uncertainty in our knowledge of the stationary averages;  (ii) no assumptions are required on the uniqueness of the stationary measures of the SDE; and (iii) the availability of high quality SDP solvers and of the modelling package GloptiPoly 3 drastically reduces the investment of effort and specialised knowledge required to implement the algorithm for the analysis of a given diffusion process.

\subsection{Problem definition.}
\label{psu} 

We consider $\rn$-valued diffusion processes $X:=\{X_t:t\geq0\}$ satisfying stochastic differential equations of the form
\begin{equation}
\label{eq:sde}
dX_t=b(X_t) \, dt+\sigma(X_t) \, dW_t,\qquad X_0=Z, 
\end{equation}
where the entries of the drift vector $b:\rn\to\rn$ and the diffusion coefficients $\sigma:\rn\to\rn\times\mathbb{R}^m$ are \emph{polynomials}. In the above, $W:=\{W_t:t\geq0\}$ is a standard $\mathbb{R}^m$-valued Brownian motion and the initial condition $Z$ is a Borel measurable random variable. We assume that the underlying filtered space $(\Omega,\cal{F},\{\cal{F}_t\}_{t\geq0},\Pb)$ satisfies the usual conditions.

A Borel probability measure $\pi$ on $\rn$ is a {\em stationary} (or {\em invariant}) {\em measure} of the dynamics \eqref{eq:sde} if 
$$
Z \sim \pi  \quad \Rightarrow \quad  X_t \sim \pi \,\,\, \mbox{for all }t\geq 0,
$$
where we use the notation $Y \sim \pi$ to mean that the random variable $Y$ has law $\pi$. The set of stationary measures of \eqref{eq:sde} is denoted $\mathcal{P}$.  

The problem we address here is how to estimate stationary averages of the form
\begin{equation}
\label{eq:stain}
\pi(f):=\int f(x) \, \pi(dx)
\end{equation}
in a systematic, computationally efficient manner. We present an algorithm that yields bounds on averages~\eqref{eq:stain} when $f$ is a \emph{polynomial}. Hence,  our algorithm can be used to bound the \emph{moments} of  the stationary measures of \eqref{eq:sde}. 

More precisely, the algorithm returns lower and upper bounds on the set  
\begin{equation}
\label{eq:set2}
B_{f,\cal{G}}^d:=\left\{\pi(f):\pi\in\mathcal{P}\text{ has finite }d\text{-th order moments and support contained in }\cal{G}\right\},
\end{equation}
where $\cal{G}$ is a given real algebraic variety
\begin{equation}\label{eq:mani}\cal{G}:=\{x\in\rn: \, g_1(x)= 0,\dots,g_\ell(x)= 0\},\end{equation}
for given polynomials $g_1,\dots,g_\ell$. The case $\cal{G}=\rn$ corresponds to $\ell=0$. 

Stationary averages of the form \eqref{eq:stain} and the set \eqref{eq:set2} are of broad interest: 
if $X$ enjoys some mild stability and regularity properties \cite{Meyn1993a}, the stationary averages~\eqref{eq:stain} provide succinct, quantitative information about the \emph{long-term} behaviour of $X$. In particular, for almost every sample path $t\mapsto X_t(\omega)$ (that is, $\Pb$-almost every $\omega \in \Omega$), there exists a $\pi\in\mathcal{P}$ such that
\begin{equation}
\label{eq:conv}
\lim_{t\to\infty}\frac{1}{t}\int_0^tf(X_s(\omega))ds=\int f(x) \, \pi(dx),
\end{equation}
see Theorem~\ref{drift} for a formal statement. Furthermore, for appropriately chosen $d$, the set $B^d_{f,\cal{G}}$ is the set of limits~\eqref{eq:conv} where $t\mapsto X_t(\omega)$ is any sample path contained in $\cal{G}$:
\begin{equation}\label{eq:bfdglim}B_{f,\cal{G}}^d=\left\{\lim_{t\to\infty}\frac{1}{t}\int_0^tf(X_s(\omega))ds:\omega\in\Omega_\cal{G}\right\},\end{equation}
where $\Omega_\cal{G}$ is the subset of samples $\omega\in\Omega$ such that $X_t(\omega)$ belongs to $\cal{G}$ for all $t\geq0$, see \cite[\S 8]{Meyn1993a}. Similar considerations also apply to the ensemble averages $\Ebb{f(X_t)}$  (see the remark after Theorem~\ref{drift}). Therefore, bounds on the set $B^d_{f,\cal{G}}$ equip us with quantitative information on the long-term behaviour of the paths of $X$ that are contained in $\cal{G}$. 
\subsection*{Remark (Linking the long-term behaviour to the stationary measures):} Although our work is motivated by the study of the long-term behaviour of a diffusion $X$, the scope of this paper is restricted to the problem of bounding the stationary averages~\eqref{eq:stain} and the associated set~\eqref{eq:set2}. 
It is important to remark that to connect the long-term behaviour of $X$ with 
the set $B^d_{f,\cal{G}}$ (and the bounds our algorithm returns) it is necessary to establish separately:
\begin{itemize}
\item[(a)] the existence of stationary measures of~\eqref{eq:sde} with support contained in $\cal{G}$ and the finiteness of their moments up to order $d$;
\item[(b)] the convergence of the time averages, i.e. verify that the limit \eqref{eq:conv} holds for the diffusion at hand;  
\item[(c)] that, for the initial conditions of interest, $X$ takes values in $\cal{G}$. 
\end{itemize}
%
%Condition (a) ensures that the set $\cal{B}^d_{f,\cal{G}}$ is not empty, whereas (b) and (c) are required to link the behaviour of the diffusion $X$ to the set $\cal{B}^d_{f,\cal{G}}$.  
A straightforward way to verify (c) is to apply It\^o's formula and check that
$$dg_1(X_t)= dg_2(X_t) = \dots = dg_\ell(X_t) = 0.$$
If the initial condition $Z$ takes values in $\cal{G}$, then $X$ clearly takes values in $\cal{G}$ as well. Establishing (a) and (b) typically requires additional proofs beyond the scope of this paper, and has been studied extensively elsewhere (see \cite{Meyn1993a,Meyn1993b,Khasminskii2012}). 
We point out that whether or not conditions (a)-(c) hold, the algorithm still yields bounds 
on $B^d_{f,\cal{G}}$. Without proving (a)-(c), it may simply be that the set~\eqref{eq:set2} is empty, 
or that the relationships \eqref{eq:conv}  or \eqref{eq:bfdglim} do not exist.
To make the paper self-contained, we recall briefly in Section~\ref{stability} some simple conditions that we use in our later examples to establish (a) and (b). 
\subsection{Brief description of algorithm.}
\label{sec:ba}
Mathematically, our approach consists of four steps:
\begin{enumerate}
\item[(1)] We derive a finite set of linear equalities satisfied by the moments of any stationary measure of~\eqref{eq:sde} (see Lemma \ref{diff} and Section \ref{lleqs}). Such a system of equalities is often underdetermined (see Example \ref{x3}) and therefore admits infinitely many solutions.  
\item[(2)] \label{item2}
To rule out spurious solutions, we exploit the fact that the solutions must be the moments of a probability measure, hence must satisfy extra constraints (e.g. even moments cannot be negative). We use well known results~\cite{Lasserre2009} to construct semidefinite inequalities (so called \emph{moment constraints}) that are satisfied by moments of any probability measure with support contained in $\cal{G}$. This is done in Section \ref{sdineqs}.
\item[(3)] Exploiting the fact that $f$ is a polynomial, we rewrite the stationary average $\pi(f)$ as a linear combination of the moments of $\pi$.
\item[(4)] By maximising (resp. minimising) the linear combination over the set of all vectors of real numbers that satisfy \emph{both} the linear equalities and the semidefinite inequalities, we obtain an upper (resp. lower) bound 
on the set $B_{f,\cal{G}}^d$.
\end{enumerate}

Computationally, to find the upper (or lower) bound we solve a semidefinite program, a particularly tractable class of convex optimisation problems for which high-quality solvers are freely available online. The semidefinite constraints in (2), popularised by Lasserre and co-authors (see \cite{Lasserre2009} and references therein), can be implemented via the freely-available, user-friendly package GloptiPoly~3 \cite{Henrion2009}, which makes the approach described here accessible to non-experts. 

\subsection*{Remark (Convergence of the algorithm):} Concerning moment approaches, a question that often arises is that of convergence of the algorithm. In the context of this paper, this question takes the following form.
Suppose that we want to obtain bounds on the set
$$B^{\infty}_{f,\cal{G}}:=\bigcap_{d=d_f}^\infty B^d_{f,\cal{G}},$$
where $d_f$ is the degree of $f$. Note that if $\cal{G}$ is compact, then any measure with support on $\mathcal{G}$ has all moments finite, thus $B^{d_f}_{f,\cal{G}}=B^{d_f+1}_{f,\cal{G}}=\dots=B^\infty_{f,\cal{G}}$; hence $B^\infty_{f,\cal{G}}$ is the only set of interest. As will become clear later, repeated applications of our algorithm yield both an increasing sequence of lower bounds and a decreasing sequence of upper bounds on $B^\infty_{f,\cal{G}}$.  The algorithm is said to converge if these sequences converge to the infimum and supremum, respectively, of $B^\infty_{f,\cal{G}}$. In general, the algorithm presented in this paper is not guaranteed to converge in the above sense. 
However, our numerics indicate that the algorithm often converges in practice (see Examples \ref{x32}, \ref{pest} and \ref{stabuns}). In Section \ref{sec42}, we do prove convergence for SDEs related to the Lyapunov exponents of linear ordinary differential equations (ODEs) perturbed by multiplicative white noise. The question of convergence is of theoretical interest, but regardless of its answer, the bounds computed are still valid and are often still appropriate for the application in question (e.g., see the examples in Section~\ref{applications}).
\subsection{Related literature and contributions of the paper.}\label{merits} 
Computational methods that yield bounds on functionals of Markov processes by combining linear equalities (arising from the definitions of the functional and the process) and moment constraints have appeared in the literature. We refer to this class of methods as \emph{generalised moment approaches} (GMAs). The various GMAs differ in the type of Markov processes and moment constraints they consider. The ideas underlying GMAs were first discussed in~\cite{Dawson1980,Dawson1980a} where they were used to obtain analytical bounds on moments for measure-valued diffusions. The first GMA was presented in E. Schwerer's PhD thesis \cite{Schwerer1996} in the context of jump processes with bounded generators and reflected Brownian motion on the non-negative orthant.  In~\cite[\S 12.4]{Hernandez-Lerma2003}, the authors present a GMA that yields bounds on the moments of stationary measures of discrete-time Feller Markov chains. In~\cite{Helmes2003}, analogous techniques are used to bound the moments of the stationary measures of diffusion approximations of the Wright-Fisher model on the unit simplex. GMAs have also been proposed to solve optimal control problems \cite{Hernandez-Lerma1999,Helmes2000}, to estimate exit times \cite{Helmes2001,Lasserre2004}, and to price financial derivatives \cite{Lasserre2006,Eriksson2011}. 

The contributions of this paper are as follows. Whereas~\cite{Schwerer1996,Helmes2003,Helmes2000} consider only stationary averages of specific SDEs, we introduce GMAs to the setting of general polynomial diffusions over unbounded domains---this requires setting up the technical background of Lemma~\ref{eqs}. Also in contrast with~~\cite{Schwerer1996,Helmes2003,Helmes2000}, we employ moment constraints that lead to SDPs instead of linear programs. In Section \ref{sec42}, we prove convergence of our algorithm for SDEs related to the Lyapunov exponents of linear ODEs perturbed by multiplicative white noise. To the best of our knowledge, this is the first such result in the setting of stationary averages of diffusion processes. The remaining contributions are the applications of the algorithm to several examples of interest. Section \ref{sec41} explains how the algorithm can be combined with the ideas underlying the Metropolis Adjusted Langevin Algorithm (a Markov chain Monte Carlo algorithm) to carry out numerical integration with respect to certain target measures, and then applies it to a simple Bayesian inference problem. In Section \ref{sec42}, we use our algorithm to obtain bounds on the Lyapunov exponents of linear differential equations perturbed by multiplicative white noise, and we show that in this case our approach is both sufficient and necessary (i.e., with enough computation power, it yields lower (upper) bounds arbitrarily close to the minimum (maximum) Lyapunov exponent, see Theorem \ref{convl}). Finally, in Section \ref{sec43} we explain how the algorithm can be extended to yield bounds on stationary averages $\pi(f)$ where $f$ is piecewise polynomial, and we use this extension to tackle a reliability problem from structural mechanics. 
\section{Background and preliminaries.}  
\subsection{Notation}
\label{prelim}
Throughout this paper we use the following notation:
\begin{itemize}
\item $\Eb[\cdot]$ denotes expectation with respect to the underlying probability measure $\Pb$, and we use $X_t$ and $X(t)$ interchangeably.
\item Given a function $h:\rn\to\r$, $\partial_ih$ denotes its partial derivative with respect to its $i^{th}$ argument; $\partial_{ij}h:=\partial_i\partial_jh$; $\nabla h$ denotes its gradient vector; and $\nabla^2 h$ denotes its Hessian matrix. 
\item Suppose that $\cal{M}$ is a smooth manifold. $C^2(\cal{M})$ denotes the set of real-valued, twice continuously differentiable functions on $\cal{M}$.
\item For any two matrices $A,B\in\rnm$, 
$$\iprod{A}{B}:=\sum_{i=1}^n\sum_{j=1}^mA_{ij}B_{ij}$$
denotes the trace inner product of $A$ and $B$, and $\norm{A}:=\sqrt{\iprod{A}{A}}$ denotes the Frobenius norm of $A$. 
\item Let $\nn$ be the set of $n$-tuples $\alpha:=(\alpha_1,\dots,\alpha_n)$ of natural numbers $\alpha_i\in\n$. Let $\nnd$ be the subset of $n$-tuples such that $|\alpha|:=\alpha_1+\dots+\alpha_n\leq d$. The cardinality of $\nnd$ is $r(d):={n+d \choose d}$.
We define the sum of two tuples $\alpha,\beta\in\nnd$ to be the tuple $\alpha+\beta:=(\alpha_1+\beta_1,\dots,\alpha_n+\beta_n)$.  
\item The space of real-valued vectors indexed by $\nnd$, $\{y:y_\alpha\in\r, \alpha\in\nnd\}$, is isomorphic to $\rrd$ and we make no distinction between them. Similarly for the space of real-valued matrices indexed by $\nnd$, $\{M:M_{\alpha\beta}\in\r,\alpha,\beta\in\nnd\}$, and $\rrdrd$. With this in mind, we denote the standard inner product on $\rrd$ as
$$\iprod{y}{z}:=\sum_{\alpha\in\nnd}y_\alpha z_\alpha,\qquad y,z\in\rrd,$$
%
%that on $\rrd$ by,
%%
%$$\iprod{M}{N}:=\sum_{\alpha\in\nnd}\sum_{\beta\in\nnd}M_{\alpha\beta}N_{\alpha\beta},\qquad \forall y,z\in\rrd$$
%%
and the standard outer product on $\rrd$ as
$$(y\otimes z)w:=\iprod{z}{w}y,\qquad y,z,w\in\rrd.$$
\item Let $\alpha\in\nnd$ and $x\in\rn$.  Monomials are denoted as $x^\alpha := \prod_{i=1}^{n} x_i^{\alpha_i}$, and 
$m_d(x):\rn\to\rrd$ denotes the vector of monomials of degree $d$ or less. 
Hence the $\alpha$-th component of the $r(d)$-dimensional vector $m_d(x)$ is given by
$$(m_d(x))_\alpha :=x^\alpha.$$
Let $\pol_d$ denote the vector space of real polynomials on $\rn$ of degree at most $d$. The set $\{x^\alpha:\alpha\in\nn_d\}$ is a basis (known as the canonical or monomial basis) of $\pol_d$, and so we can write any polynomial $p\in\pol_d$ as
$$p(x) = \sum_{\alpha\in \nn_d}\pmb{p}_\alpha x^\alpha = \iprod{\pmb{p}}{m_d(x)},\quad\forall x\in\rn,$$
where $\pmb{p}:=(\pmb{p}_\alpha)_{\alpha\in\nnd}$ in $\rrd$ is the \emph{vector of coefficients} of $p$. 
%and we have that 
%%
%$$p(x)=\iprod{\pmb{p}}{m_d(x)},\quad\forall x\in\rn,$$
%%
%%or in short
%%
%$$p := \iprod{\pmb{p}}{m_d}.$$
%
We denote the degree of any polynomial $p$ with $d_p$.
\item Let $\mu$ be a Borel measure on $\rn$. A vector $y$ in $\rrd$ is the \emph{vector of moments of order up to $d$} of $\mu$ if, for any $\alpha\in\nnd$, the $\alpha$-th component of $y$ is given by
$$y_\alpha=\mu(x^\alpha):=\int_{\rn} x^{\alpha}\mu(dx),$$
assuming that the integrals are well defined.
\end{itemize}

\subsection{Stability and regularity properties of $X$.}\label{stability} We now briefly review well known properties of $X$ and a relevant drift condition used in the examples below. Throughout this section we assume that $\cal{G}$ is an $(n-\ell)$-dimensional smooth submanifold of $\rn$. For this to be the case, it is sufficient that the vectors $\nabla g_1(x),\nabla g_2(x),\dots,\nabla g_\ell(x)$ form a linearly independent set, for each $x$ in $\cal{G}$, where the $g_i$'s are defined in~\eqref{eq:mani}. 

Firstly, smoothness of the components of $b$ and $\sigma$ imply that \eqref{eq:sde} has a unique strong solution $X$, which is defined up to a stopping time $\tau_\infty$~\cite{Watanabe1989}. The \emph{generator} (or Kolmogorov operator) $\cal{A}$ associated with \eqref{eq:sde} is the second order differential operator
\begin{align*}
\cal{A}h(x):=\iprod{\nabla h(x)}{b(x)}+\frac{1}{2}\iprod{\nabla^2h(x)}{a(x)} =\sum_{i=1}^nb_i(x)\partial_ih(x)+\frac{1}{2}\sum_{i,j=1}^na_{ij}(x)\partial_{ij}h(x),
\end{align*}
for $h\in C^2(\cal{G})$, $x\in\cal{G}$, and where $a:=\sigma\sigma^T$ denotes the diffusion matrix of \eqref{eq:sde}.  It is well known that if $\cal{A}$ is a hypoelliptic operator on $C^2(\cal{G})$ (see \cite{Hormander1967}), then \eqref{eq:sde} generates a strong Feller Markov process. This is the regularity property of $X$ we use in our examples below. Although this condition can be replaced with weaker ones (e.g., $X$ being a T-process \cite{Meyn1993a}), such alternative conditions usually require more work to establish in practice. The stability properties we require are summarised as follows: 
\begin{condition}[\cite{Meyn1993a,Meyn1993b}]\label{small} Suppose that the paths of the diffusion $X$ are contained in $\cal{G}$ (that is, $\Pb(\{X_t\in\cal{G},\smallskip\forall\smallskip 0\leq t\leq \tau_\infty\})=1$), and that either one of the following conditions holds:
\begin{enumerate}[label=(\roman*)]
\item the manifold $\cal{G}$ is compact.
\item (Drift condition) there exists a function $u\in C^2(\cal{G})$ and a constant $c>0$ such that for each $q\in\r$ the sub-level set
$$\{x\in\cal{G}:u(x)\leq q\}$$
is compact and such that
$$\cal{A}u(x)\leq -cu(x), $$
holds for all $x$ in $\cal{G}$ except those in a compact subset of $\cal{G}$.
\end{enumerate}
\end{condition}
\begin{theorem}[\cite{Meyn1993a,Meyn1993b}]\label{drift}
Suppose that $\cal{A}$ is hypoelliptic on $C^2(\cal{G})$ and that Condition \ref{small} holds. Then:
\begin{enumerate}[label*=\textit{(}\roman*\textit{)}] 
\item The solution of \eqref{eq:sde} is globally defined, that is, $\Pb(\{\tau_\infty=\infty\})=1$.  
\item The SDE \eqref{eq:sde} has at least one stationary measure with support contained in $\mathcal{G}$.
\item If Condition~\ref{small} \textit{(i)} holds, then for any measurable $f$ and almost every sample path $t\mapsto X_t(\omega)$, the limit \eqref{eq:conv} holds, where $\pi\in\mathcal{P}$ has support contained in $\mathcal{G}$ and depends on the starting position of the path, $Z(\omega)$. Furthermore, equation \eqref{eq:bfdglim} holds.
\item If $\cal{G}$ is not compact but Condition~\ref{small} \textit{(ii)} holds, then $\pi(u)<\infty$ for any $\pi\in\mathcal{P}$ with support contained in $\cal{G}$, and the same as in Theorem~\ref{drift} \textit{(iii)} is true for every measurable $f$ such that
$$\sup_{x\in\cal{G}}\left(\frac{\mmag{f(x)}}{1+\mmag{u(x)}}\right)<\infty.$$
\end{enumerate}
\end{theorem}
\begin{proof} In the case $(i)$, it is easy to argue that the solution is globally defined. In the case $(ii)$, global existence of the solution follows from \cite[Thrm.2.1]{Meyn1993b}. The rest follows from Theorems 3.4 and 8.1 in \cite{Meyn1993a}, plus Theorem 4.7 in \cite{Meyn1993b} in the case of the drift condition.
\end{proof}
\subsection*{Remark (The ensemble averages):} If, additionally to the premise of Theorem \ref{drift}, the semigroup generated by~\eqref{eq:sde} is aperiodic \cite{Meyn1993a}, then the analogous statements to Theorem~\ref{drift} \textit{(iii)} and \textit{(iv)} 
hold for the ensemble averages $\Ebb{f(X_t)}$. In this case, we have that
$$
\lim_{t\to\infty} \Eb[f(X_t)]= \int f(x) \, \tilde{\pi}(dx),
$$
where $\tilde{\pi}\in\cal{P}$ depends on the law of the initial condition $Z$, see \cite[Thrm.8.1]{Meyn1993a}.
\subsection{A relationship between the generator of \eqref{eq:sde} and its stationary measures.}\label{eqs}

The following technical lemma is necessary for the development of the algorithm presented in this paper.

An application of It\^o's formula shows that, if $h\in C^2(\rn)$, then
\begin{equation}
\label{eq:mart}
D_t^h:=h(X_t)-h(X_0)-\int_0^t\cal{A}h(X_s)ds,
\end{equation}
is a local martingale. 
The generator $\cal{A}$ evaluated at function $h$ and point $x$ describes the rate of change in time of the expected value of $h(X_t)$ conditioned on the event $\{X_t=x\}$. 
If $\pi$ is the law of $X_t$, then $\pi(\cal{A}h)$ describes the rate of change in time of $\Eb[h(X_t)]$. 
It follows that if $\pi$ is a stationary measure of \eqref{eq:sde}, then the law of $X_t$ does not change in time, 
and thus we would expect that $\pi(\cal{A}h)=0$. Unfortunately, for technical reasons, this is not always the case (see \cite{Glynn2008} for counterexamples). However, it is not difficult to find sufficient conditions on $h$ 
such that $\pi(\cal{A}h)=0$. 
The following lemma gives one such condition specialised for polynomial functions $h$, which are the focus of this paper. For a proof of the lemma see Appendix \ref{appendix1}.
\begin{lemma}\label{diff}
Let $\pi$ be a stationary measure of \eqref{eq:sde} whose moments of order $d$ exist and are finite. 
Let $h$ be polynomial with degree $d_h\leq d-\max_{i,j}\{d_{b_i},d_{a_{ij}}\}$, where $b$ is the drift vector 
and $a:=\sigma\sigma^T$ is the diffusion matrix of~\eqref{eq:sde}. Then
$$\pi(\cal{A}h)=0.$$
\end{lemma}
\section{The algorithm.}\label{method} Our algorithm constructs a tractable \emph{outer approximation} $C^d_{f,\cal{G}}$ of the set $B^d_{f,\cal{G}}$ defined by \eqref{eq:set2}. As shown below, the approximation we derive is the image of a \emph{spectrahedron} (a set defined by linear equalities and semidefinite inequalities) through a linear functional. Finding the infimum and supremum of $C^d_{f,\cal{G}}$ reduces to solving two SDPs, which can be efficiently carried out using one of several high-quality solvers freely available. Since $B^d_{f,\cal{G}} \subseteq C^d_{f,\cal{G}}$, the computed infimum (supremum) is a lower (upper) bound of $B^d_{f,\cal{G}}$. 

To introduce our method, we first take a closer look at the set $B^d_{f,\cal{G}}$ from two alternative perspectives. First, note that the set $B^d_{f,\cal{G}}$ defined in~\eqref{eq:set2} is the image of the set
$$\mathcal{P}^d_\cal{G}:= \{\pi\in\mathcal{P}:\pi  \text{ has finite }d\text{-th order moments and support contained in }\cal{G}\}$$
through the linear functional (on the vector space of signed measures) $\pi\mapsto\pi(f)$.
It is straightforward to check that, as a subset of the vector space of signed measures, $\mathcal{P}^d_\cal{G}$ is convex. Consequently, its image $B^d_{f,\cal{G}}$ is a (possibly unbounded) interval, which is fully described  
%$B^d_{f,\cal{G}}$ 
by its supremum and infimum (leaving aside whether or not $B^d_{f,\cal{G}}$ contains its endpoints).  

Alternatively, since $f$ is a polynomial (with vector of coefficients $\pmb{f}$), the set $B^d_{f,\cal{G}}$ is the image of the set
$$\cal{Y}^d_\cal{G}:=\{\text{$y \in \rrd$: $y$ is a vector of moments up to order $d$ of a measure in $\mathcal{P}^d_\cal{G}$}\}$$ 
through the linear functional (on $\rrd$) $y\mapsto\iprod{y}{\pmb{f}}.$
We now see some concrete examples of these sets.
\begin{example} To introduce our ideas, we use an example for which there is an extensive body of results. 
Consider the two-dimensional SDE
\begin{equation}\label{eq:brwncirc}dX_t=-\frac{1}{2}X_tdt+\begin{bmatrix}0&-1\\1&0\end{bmatrix}X_tdW_t,\qquad X_0=Z.\end{equation}
Applying It\^o's formula gives $d\norm{X_t}\equiv0$; hence, $\norm{X_t}=\norm{X_0}, \forall t\geq0$. If the initial condition $Z$ takes values in the circle of radius $R$, $\mathbb{S}^1_R$, then the paths of $X$ remain in $\mathbb{S}^1_R$.  
Using H\"ormander's condition~\cite{Hormander1967} it is easy to verify that, for each $R>0$, the generator $\cal{A}$ is a hypoelliptic operator on $C^2(\mathbb{S}^1_R)$. By compactness of $\mathbb{S}^1_R$, Condition \ref{small} $(i)$ is satisfied, and Theorem \ref{drift} states that, for each $R>0$, \eqref{eq:brwncirc} has at least one stationary measure with support contained in $\mathbb{S}^1_R$. It is also well known~\cite{Khasminskii1967} that for each $R$, \eqref{eq:brwncirc} has only one such measure $\pi_R$, which is the uniform distribution on $\mathbb{S}^1_R$. 
Therefore, for a given $R$ and $f$ 
\begin{equation}
\label{eq:ps}
\pi_R(f)=\frac{1}{2\pi}\int_0^{2\pi}f(R\cos(\theta),R\sin(\theta))d\theta.
\end{equation}
Note that if $R=0$, then $X \equiv 0$, hence 
$\pi_0:=\delta_0$, the Dirac measure at zero, is also a stationary measure of \eqref{eq:brwncirc}. 
Consequently, for any $d$:
$$\mathcal{P}^d_{\mathbb{S}^1_R}=\{\pi_R\}, \quad \mathcal{P}^d_{\r^2} = \{\pi_R:R\geq0\}, \quad 
\mathcal{P}^d_{\cal{G}}=\emptyset\quad \text{for all other varieties $\cal{G}$}.$$
By the symmetry of the measure $\pi_R$, it is easy to show that 
%$\pi_R(x_1^{\alpha_1}x_2^{\alpha_2})$ is zero unless both $\alpha_1$ and $\alpha_2$ are even numbers, in which case it is not hard to show that
%
\begin{align*}
y_{({\alpha_1},{\alpha_2})} := \pi_R \left(x_1^{\alpha_1}x_2^{\alpha_2}\right)= 
\begin{cases} 
\dfrac{R^2 \, \pi^{3/2} \, 2^{\mmag{\alpha}}}{\mmag{\alpha}! \,\, \Gamma \left(\frac{1-\alpha_1}{2}\right) \Gamma \left(\frac{1-\alpha_2}{2}\right) \Gamma \left(\frac{1-\mmag{\alpha}}{2}\right)} 
\quad &\text{if $\alpha_1$ and $\alpha_2$ are even} \\
0 &\text{otherwise},
\end{cases}
\end{align*}
where $\Gamma(\cdot)$ denotes the gamma function. It is now straightforward to describe the sets $\cal{Y}^d_\cal{G}$. For instance, consider the sets containing the $r(2)$-dimensional vectors of moments up to order $d=2$ defined as
$y: = (y_{(0,0)},y_{(1,0)},y_{(0,1)},y_{(2,0)},y_{(1,1)},y_{(0,2)}) \in \mathbb{R}^6$.  
Then we have
$$\mathcal{Y}^2_{\mathbb{S}^1_R}=\left \{ \left(1,0,0,\frac{R^2}{2},0,\frac{R^2}{2}\right) \right\} 
\quad \text{and} \quad 
\mathcal{Y}^2_{\r^2}=\{1\}\times\{0\}\times\{0\}\times[0,\infty)\times\{0\}\times[0,\infty).$$
Using the above descriptions of the sets $\cal{P}^d_\cal{G}$ together with the map $\pi\mapsto\pi(f)$ or, alternatively, the sets $\cal{Y}^d_\cal{G}$ together with the map $y\mapsto\iprod{y}{\pmb{f}}$, we can deduce any projection of interest $B^d_{f,\cal{G}}$. For instance,
\begin{align*}
\text{for} \quad f_1(x)&=x_1x_2,     & B_{f_1,\mathbb{S}^1_R}^d & =\{0\} &\text{and}&   & B_{f_1,\r^2}^d &= \{0\} \\
\text{for} \quad f_2(x)&=x_2^2+1,   & B_{f_2,\mathbb{S}^1_R}^d &=\left \{\dfrac{R^2}{2}+1 \right \} &\text{and}& & B_{f_2,\r^2}^d& =[1,\infty) \\
\text{for} \quad f_3(x)&=x_2-2x^2_1+3,   & B_{f_3,\mathbb{S}^1_R}^d &= \left \{3-R^2 \right \} &\text{and}& &B_{f_3,\r^2}^d&=(-\infty,3].
\end{align*}
\end{example}

In the above example, we could obtain the sets $\mathcal{P}^d_\cal{G}$, $\mathcal{Y}^d_\cal{G}$ and their projections $B^d_{f,\cal{G}}$ directly from~\eqref{eq:sde}. However, this is difficult in general. 
Indeed, results from real algebraic geometry show that optimising over the cone of vectors whose components are moments of a measure is an NP-hard problem \cite{Lasserre2009}. We believe that, except for trivial cases, the same holds for $\cal{Y}^d_\cal{G}$ which is a subset of this cone. Instead, we construct here a spectrahedral outer approximation 
$\cal{O}^d_\cal{G}$ of the set $\cal{Y}^d_\cal{G}$.  
Optimising over $\cal{O}^d_\cal{G}$ consists of solving an SDP, a polynomial-time problem. Explicitly, $\cal{O}^d_\cal{G}$ is a subset of $\rrd$ defined by linear equalities and semidefinite inequalities such that 
$\cal{Y}^d_\cal{G} \subseteq \cal{O}^d_\cal{G}$. 
Because the outer approximation is a convex set, its image through the linear functional $y\mapsto\iprod{y}{\pmb{f}}$
is an interval that contains $B^d_{f,\cal{G}}$:
\begin{equation}
\label{eq:proj}
B^d_{f,\cal{G}} \subseteq C^d_{f,\cal{G}}:=\left\{\iprod{y}{\pmb{f}}:y\in\cal{O}^d_\cal{G}\right\}.
\end{equation}
%
 %
%
%We can then find the infimum (or the supremum) of $C^d_{f,\cal{G}}$ by solving a semidefinite program, for which excellent solvers are freely available online. Since $B^d_{f,\cal{G}}\subseteq C^d_{f,\cal{G}}$ this infimum (resp. supremum) is a lower (resp. upper) bound of $B^d_{f,\cal{G}}$.

Hence our task is reduced to obtaining the outer approximation $\cal{O}^d_\cal{G}$. We do this in two steps:
\begin{itemize}
\item In Section \ref{lleqs}, we use Lemma \ref{diff} to construct a set of linear equalities satisfied by the moments of any stationary measure of \eqref{eq:sde}.
\item In Section \ref{sdineqs}, we construct a set of linear equalities and a semidefinite inequality satisfied by the moments of any unsigned measure with support on $\cal{G}$.
\end{itemize} 
The outer approximation $\cal{O}^d_\cal{G}$ then consists of the set of vectors in $\rrd$ that satisfy both of the above.

\subsection{Linear equalities satisfied by the moments of stationary measures.}
\label{lleqs}
By assumption, both the drift vector and the diffusion coefficients in~\eqref{eq:sde} are polynomials. Therefore, if $h$ is a polynomial, $\cal{A}h$ is also a polynomial. Suppose that $y$ belongs to $\cal{Y}^d_\cal{G}$ and choose any measure $\pi$ in $\mathcal{P}^d_\cal{G}$ that has $y$ as its vector of moments of order $d$. From Lemma~\ref{diff}, if $d\geq d_\mathcal{A}:=\max\{d_{b_i},d_{a_{ij}}\}$, then
$$\iprod{y}{\pmb{\cal{A}h}}=  \sum_{\beta\in\nnd}\pmb{(\cal{A}h)}_\beta \, y_\beta=
\sum_{\beta\in\nnd}\pmb{(\cal{A}h)}_\beta \, \pi(x^\beta)= \pi\Big(\sum_{\beta\in\nnd}\pmb{(\cal{A}h)}_\beta \, x^\beta\Big)
=\pi(\cal{A}h)=0, \quad \forall h \in\pol_{d-d_\mathcal{A}}.$$
Since $\{x^\alpha:\alpha\in\nn_{d-d_\mathcal{A}}\}$ spans $\pol_{d-d_\mathcal{A}}$, checking that $y$ satisfies the above is equivalent to checking that 
\begin{equation}
\label{eq:eqs}
\iprod{y}{\pmb{\cal{A}x^\alpha}}=0,\quad \forall \alpha\in \nn_{d-d_\mathcal{A}}.
\end{equation}
In words, every vector in $\cal{Y}^d_\cal{G}$ satisfies the $r(d-d_\mathcal{A})$ linear equalities defined by~\eqref{eq:eqs}. 

At this point, it is worth remarking that the conditions  on $h$ and $\pi$ in Lemma~\ref{eqs} are only sufficient but not necessary for $\pi(\cal{A}h)=0$ to hold, as the following example shows.
\begin{example}\label{GBM} Let $\lambda$ be a positive even integer, and consider the one-dimensional SDE
\begin{equation}\label{eq:GBM}dX_t=(1-\lambda X_t)dt+\sqrt{2}X_tdW_t,\qquad X_0=Z.\end{equation}
%
%A typical analysis of the SDE would be along the lines: It is straightforward to check that $\cal{A}$ is hypoelliptic, and thus any stationary measure of $X$ has a smooth density with respect to the Lebesgue measure, see \cite[Sec.7]{Bellet2006}. Furthermore, $\rho:\rn\to[0,\infty)$ is a density of an stationary measure of $X$ if\footnote{{\color{red} Michela: Is this even true? Do we not need boundedness assumptions on the drift and diffusion coefficients?}} and only if it satisfies $\cal{A}^*\rho=0$, where $\cal{A}^*$ denotes the adjoint (in $L^2$) of $\cal{A}$. It is a simple matter to verify that the inverse Gamma distribution with shape parameter $\lambda+1$ and unit scale parameter, 
%%
%\begin{equation}\label{eq:density}\rho(x)=\left\{\begin{array}{l l}\frac{1}{\lambda!}e^{-\frac{1}{x}}x^{-(\lambda+2)}&\text{if }x>0\\0&\text{if }x\leq 0\end{array}\right.,\end{equation}
%%
%satisfies $\cal{A}^*\rho=0$. Thus the above is the density of an invariant probability measure $\pi$. Employing the Support Theorem of Stroock and Varadhan \cite[Thrm.6.6]{Bellet2006}, we can also argue that the stationary measure is unique. 
%
It is straightforward to verify that Condition \ref{small} $(ii)$ holds with $u(x):=x^\lambda$ and $\cal{G}:=\r$, and to use 
H\"ormander's condition to establish that the generator of \eqref{eq:GBM} is hypoelliptic on $C^2(\r)$.  From Theorem \ref{drift}, it follows that: \eqref{eq:GBM} has at least one stationary measure; that all of its stationary measures have moments up to order $\lambda$; and that \eqref{eq:conv} holds for any $f\in\pol_\lambda$. From \eqref{eq:eqs}, 
we deduce that given the moments of any such stationary measure $y\in \r^{\lambda+1}$, then
\begin{equation}
\label{eq:GBMeq} 
k \,(y_{k-1}-(\lambda+1-k)y_k)=0,\qquad k=1,2,\dots,\lambda-2.
\end{equation}
We are only interested in solutions to these equations that the are moments of a probability measure. Hence we can append the normalisation $y_0=1$ and solve~\eqref{eq:GBMeq} to obtain
\begin{equation}
\label{eq:invagammon}
y_k=\prod_{j=0}^{k-1}\frac{1}{\lambda-j}.
\end{equation}
%
%
%We can then compute the value of the integral $\pi(f)$ for any $f\in\pol_{\lambda-2}$ using \eqref{eq:invagammon}, where $\pi$ is any stationary measure of \eqref{eq:GBM}. 
%
In fact, \eqref{eq:invagammon} holds for $k=1,\dots,\lambda$ (instead of only for $k=1,\dots,\lambda-2$). 
This is easily deduced by solving the Fokker-Planck equation associated with \eqref{eq:GBM} and showing that the density of a stationary measure of \eqref{eq:GBM} is given by the inverse Gamma distribution
$\frac{1}{\lambda!}x^{-\lambda-2}e^{-\frac{1}{x}}$. 
%the inverse Gamma distribution with shape parameter $\lambda+1$ and unit scale parameter.
%
Indeed, the moments of this distribution are given by \eqref{eq:invagammon} for $k=1,\dots,\lambda$. 
Additionally employing the Support Theorem of Stroock and Varadhan~\cite[Thrm.6.6]{Bellet2006}, we can conclude that this is the only stationary measure of \eqref{eq:GBM}. In conclusion, the moments of any stationary measure of \eqref{eq:GBM} satisfy \eqref{eq:GBMeq} for $k=1,\dots,\lambda$ although Lemma~\ref{diff} only implies that they are satisfied for $k=1,\dots,\lambda-2$.
%
%Employing the Support Theorem of Stroock and Varadhan \cite[Thrm.6.6]{Bellet2006}, we can also argue that there is only one stationary measure. The moments of 
%
%which are indeed the moments of the inverse Gamma distribution \eqref{eq:density}. Notice that the above formula makes no sense if we pick $k\geq\lambda+1$. Thus, Lemma \ref{diff} tells us that $X$ has no stationary measure with moments of order greater than $\lambda$ and the first few moments of any stationary probability measure of $X$ are those given by \eqref{eq:invagammon}. The above coupled with a drift condition, for example \cite[Thrm.6.1]{Meyn1993b} with $V(x):=x^k$ and $k<\lambda+1$, would specify the long-term behaviour of $f\circ X$ for any polynomial $f$ (in the sense of \eqref{eq:conv}), without requiring any of the analysis at the beginning of the example.
\end{example}

For most SDEs, $y_0=1$ together with the equations \eqref{eq:eqs} defines a set of underdetermined linear equations as the following example illustrates.
\begin{example}\label{x3}Consider the SDE
\begin{equation}\label{eq:x3}dX_t=(1-2X_t^3)dt+\sqrt{2}X_tdW_t,\qquad X_0=Z.\end{equation}
It is straightforward to verify that Condition~\ref{small} $(ii)$ is satisfied with $u(x):=e^{x^2/2}$ and $\cal{G}:=\r$, and to use 
H\"ormander's condition to establish that the generator of \eqref{eq:x3} is hypoelliptic on $C^2(\r)$. Theorem \ref{drift} then establishes that \eqref{eq:x3} has globally defined solutions; that it has at least one stationary measure; that all stationary measures have all moments finite; and that \eqref{eq:conv} holds for any $f\in\pol$.

Equations~\eqref{eq:eqs} in this case read:
$$
k \left(y_{k-1}+(k-1)y_k-2y_{k+2} \right) = 0, \qquad  k=1,2,\dots.$$
By appending $y_0=1$ to the above, we can only solve for $y_3=y_0/2=1/2$.
The set of equations formed by the first $\tilde{k} \geq 2$ conditions (together with $y_0=1$) is underdetermined, and no other moment can be solved for. %To solve this problem we need to introduce further constraints, as shown in the following section.
\end{example}

\subsection{Moment conditions.}\label{sdineqs}
That the moment equations \eqref{eq:eqs} are underdetermined in the above example is essentially the same issue that \emph{moment closure methods} attempt to address (see \cite[\S 3.4]{Schueller1997} for a review). These methods `close' the equations \eqref{eq:eqs} by heuristically removing the dependence of the first few equations on higher moments.
We do not follow this approach here. Instead, we overcome this issue by exploiting the fact that we are not interested in \emph{all} the solutions to the equations \eqref{eq:eqs}, but only in those that are the vector of moments of a probability measure with support contained in $\cal{G}$. 
There are various tractable conditions, known as \emph{moment conditions}, which are satisfied by the vector of moments of any measure with support contained in $\cal{G}$, but not in general by an arbitrary vector of real numbers. For example, trivially, the even moments of an unsigned Borel measure on $\rn$ must be non-negative. We now describe in more detail the moment conditions we use in the rest of this paper.

Let $\pi$ be a measure with vector of moments $y$ of order $d$, and let $g$ be a polynomial of degree $d$. If $\pi$ has support contained in $\{x\in\rn:g(x)=0\}$, then
$$
\iprod{y}{\pmb{g}}=\sum_{\alpha\in\nnd}\pmb{g}_\alpha y_\alpha= \sum_{\alpha\in\nnd}\pmb{g}_\alpha \pi(x^\alpha)= \pi\Big(\sum_{\alpha\in\nnd}\pmb{g}_\alpha x^\alpha\Big) =\pi(g)=0.
$$
By the definition \eqref{eq:mani}, $g_j$ is zero everywhere in $\cal{G}$ for all $j=1,2,\dots,\ell$. Thus
$\iprod{y}{\pmb{g_jh}}=0$
for every $j$ and $h\in \pol_{d-d_{g_j}}$. 
Since $\{x^\alpha:\alpha\in\nn_{d-d_j}\}$ spans $\pol_{d-d_j}$, checking that $y$ satisfies $\iprod{y}{\pmb{g_jh}}=0$ for any given $h\in \pol_{d-d_{g_j}}$ is equivalent to checking that $y$ satisfies the following $r(d-d_{g_j})$ equations:
\begin{equation}
\label{eq:t2}
\iprod{y}{\pmb{g_jx^\alpha}}=0,\qquad \forall\alpha\in \nn_{d-d_{g_j}}.
\end{equation}
%
%By choosing several monomials $h(x):=x^\alpha$ we can systematically construct a set of linear equations that must be satisfied by any measure (with sufficient finite moments) that has support contained in $\cal{G}$.
%
To these linear equations, we append a semidefinite inequality that stems from the fact that probability measures are unsigned. 
Let $h$ be any polynomial of degree $s(d):=\lfloor d/2\rfloor$. Then it follows that
$$(h(x))^2=\iprod{m_{s(d)}(x)}{\pmb{h}}^2=\iprod{\iprod{m_{s(d)}(x)}{\pmb{h}}m_{s(d)}(x)}{\pmb{h}}=\iprod{(m_{s(d)}(x)\otimes m_{s(d)}(x))\pmb{h}}{\pmb{h}}.$$
Since $h^2$ is a non-negative function we have that
\begin{equation}
\label{eq:psd}
0 \leq \pi(h^2)= \pi \Big(\iprod{(m_{s(d)}\otimes m_{s(d)})\pmb{h}}{\pmb{h}} \Big)=
\iprod{ \pi \left(m_{s(d)}\otimes m_{s(d)}\right) \pmb{h}}{\pmb{h}}=
\iprod{M_{s(d)}(y)\pmb{h}}{\pmb{h}},
\end{equation}
where the \emph{moment matrix} 
$$M_{s(d)}(y):=\pi(m_{s(d)}\otimes m_{s(d)})\in\mathbb{R}^{r({s(d)})\times r({s(d)})}$$
denotes the element-wise integration of the matrix $m_{s(d)}\otimes m_{s(d)}$. Note that the entries of the moment matrix are a function of the moments of $\pi$:
$$\left(M_{s(d)}(y)\right)_{\alpha\beta}:=y_{\alpha+\beta},\qquad \forall \alpha,\beta\in \nn_{s(d)}.$$
Since \eqref{eq:psd} holds for all $h\in\pol_{s(d)}$, $M_{s(d)}(y)$ is positive semidefinite: 
\begin{equation}\label{eq:t1}M_{s(d)}(y)\succeq0.\end{equation}
We then combine \eqref{eq:eqs}, \eqref{eq:t2}, \eqref{eq:t1} and the normalisation $y_0=1$ to obtain our outer approximation:
\begin{equation}
\label{eq:outer}
\cal{O}^d_\cal{G}:=\left\{y\in\rrd:
%\left\{
\begin{array}{l} 
   y_0=1 \\ 
  \iprod{y}{\pmb{\cal{A}x^\alpha}}=0,\,\forall \alpha\in \nn_{d-d_\mathcal{A}}\\ 
  \iprod{y}{\pmb{g_jx^\alpha}}=0,\, \forall \alpha\in \n^n_{d-d_{g_j}},\,\forall j=1,\dots,\ell\\ 
  M_{s(d)}(y)\succeq 0 
\end{array} 
%  \right\}
  \right \}  \supseteq \cal{Y}^d_\cal{G}.
\end{equation}
From \eqref{eq:proj}, it follows that the projection $C^d_{f,\cal{G}}=\left\{\iprod{y}{\pmb{f}}:y\in\cal{O}^d_\cal{G}\right\}$ contains $B^d_{f,\cal{G}}$. Therefore, we have the following bounds: 
\begin{equation}\label{eq:bounds}\rho_{f,\cal{G}}^d:=\inf C_{f,\cal{G}}^d \leq\inf B_{f,\cal{G}}^d,\qquad \sup B_{f,\cal{G}}^d\leq\sup C_{f,\cal{G}}^d=:\eta_{f,\cal{G}}^d.\end{equation}
Computing $\rho_{f,\cal{G}}^d$ and $\eta_{f,\cal{G}}^d$ can be efficiently done by solving two SDPs with $r(d)$ variables each.
\begin{example}
\label{x32} 
Consider again SDE \eqref{eq:x3} from Example \ref{x3}, whose stationary measures have moments of all orders. Hence
$$B^1_{x,\r}=B^2_{x,\r}=\dots=B^\infty_{x,\r}.$$
To find bounds on the mean of the stationary measures of the SDE we construct the outer approximations of $B^\infty_{x,\r}$, as described above. The first few such approximations are:
$$C^1_{x,\r}=\left\{y_1:y\in\mathbb{R}^2,\quad y_0=1,\quad y_0\geq0\right\},$$ 
$$C^2_{x,\r}=\left\{y_1:y\in\mathbb{R}^3,\quad y_0=1,\quad\begin{bmatrix}y_0&y_1\\y_1&y_2\end{bmatrix}\succeq 0\right\},$$
$$C^3_{x,\r}=\left\{y_1:y\in\mathbb{R}^4,\quad y_0=1,\quad\begin{bmatrix}y_0&y_1\\y_1&y_2\end{bmatrix}\succeq 0\right\},$$
$$C^4_{x,\r}=\left\{y_1:y\in\mathbb{R}^5,\quad \begin{array}{l}y_0=1\\ y_0-y_3=0\end{array}\quad\begin{bmatrix}y_0&y_1&y_2\\y_1&y_2&y_3\\y_2&y_3&y_4\end{bmatrix}\succeq 0\right\},$$
$$C^5_{x,\r}=\left\{y_1:y\in\mathbb{R}^6,\quad \begin{array}{l}y_0=1\\y_0-y_3=0\\ y_1+y_2-y_4=0\end{array}\quad\begin{bmatrix}y_0&y_1&y_2\\y_1&y_2&y_3\\y_2&y_3&y_4\end{bmatrix}\succeq 0\right\},$$
$$C^6_{x,\r}=\left\{y_1:y\in\mathbb{R}^7,\quad \begin{array}{l}y_0=1\\y_0-y_3=0\\ y_1+y_2-y_4=0\\ y_2+2y_3-y_5=0\end{array}\quad\begin{bmatrix}y_0&y_1&y_2&y_3\\y_1&y_2&y_3&y_4\\y_2&y_3&y_4&y_5\\y_3&y_4&y_5&y_6\end{bmatrix}\succeq 0\right\}.$$
Since a matrix is positive semidefinite if and only if all its principal minors are non-negative, it follows trivially that
$C^1_{x,\r}=C^2_{x,\r}=C^3_{x,\r}=\r$. Hence optimising over these sets yields uninformative bounds: 
$\rho^1_{x,\r}=\rho^2_{x,\r}=\rho^3_{x,\r}=-\infty$ and $\eta^1_{x,\r}=\eta^2_{x,\r}=\eta^3_{x,\r}=\infty$. 
The higher order approximations $d > 4$, however, lead to non-trivial bounds (Table \ref{table1}).
To obtain the endpoints of the higher order approximations $C^4_{x,\r}, C^5_{x,\r},\dots$ we used the SDP-solver SDPT3. 
\begin{table}[!ht]
\begin{center}
\begin{tabular}{c|cccccccccc}%\hline
$d$&$\leq4$&$5$&$6$&$7$&$8$&$9$&$10$&$11$&$12$&$13$\\\hline
$\rho_{x,\r}^{d}$& $-\infty$   &  $0.4133$   &    $0.4134$  & $0.6202$ & $0.6202$   &$0.6365$    & $0.6365$   & $0.6376$ &$0.6377$&$0.6377$\\%\hline
$\eta_{x,\r}^d$  &$\infty$   & $0.8283$   & $0.8282$  &   $0.6758$   & $0.6757$    & $0.6495$ &  $0.6495$&  $0.6494$&$0.6494$&$0.6428$ \\\hline
$d$&$14$&$15$&$16$&$17$&$18$&$19$&$20$&$21$&$22$&$23$\\\hline
$\rho_{x,\r}^{d}$&$0.6377$&$0.6377$&$0.6377$&$0.6377$&$0.6377$&$0.6377$&$0.6377$&$0.6377$&$0.6377$&$0.6377$ \\
$\eta_{x,\r}^d$  &$0.6428$ &$0.6404$&$0.6404$ &$0.6402$&$0.6402$&$0.6389$&$0.6389$&$0.6387$&$0.6387$&$0.6384$\\
\end{tabular}%}
\caption{Bounds on the mean of the stationary measures of the SDE~\eqref{eq:x3}. In total, $40$ bounds were computed taking a total CPU time of $10.3$ seconds, averaging $0.26$ seconds per bound.}\vspace{-20pt}
\label{table1}
\end{center}
\end{table}

\end{example}

For many SDEs, like those in Examples \ref{GBM} and \ref{x3}, all stationary measures have moments of some order $d$. If $\cal{G}$ is compact, $d=\infty$; otherwise, such a $d$ can be found by verifying a drift condition like the one in Condition \ref{small}. In such cases, $B^k_{f,\cal{G}}=B^d_{f,\cal{G}}$ for all $d_f\leq k\leq d$. However, as Example \ref{x32} shows,  this does not hold in general for our outer approximations. Instead, we only have that $C^k_{f,\cal{G}}\supseteq C^d_{f,\cal{G}}$ for all $d_f\leq k\leq d$. Therefore
$\{\rho^k_{f,\cal{G}}:d_f\leq k\leq d\}\text{ and } \{\eta^k_{f,\cal{G}}:d_f\leq k\leq d\}$
are monotone non-decreasing (resp. non-increasing) sequences of lower (resp. upper) bounds on $B^d_{f,\cal{G}}$. 
In practice, the best bounds are obtained by solving for the infimum/supremum of $C^k_{f,\cal{G}}$ for the largest 
$d_f\leq k\leq d$ that can be handled computationally by the solver.

\section{Applications.}\label{applications} We now consider three applications of the algorithm. To compute the bounds presented in this section we used the modelling package GloptiPoly 3 \cite{Henrion2009} to construct the SDPs corresponding to the outer approximations~\eqref{eq:outer} and the solver SDPT3 \cite{Toh1999} to solve the SDPs. All computations were carried out on a desktop computer with a 3.4 GHz processor and 16GB of memory running Ubuntu 14.04.
\subsection{Langevin diffusions, numerical integration, and an inference problem.}\label{sec41}
The Metropolis Adjusted Langevin Algorithm (MALA)~\cite{Roberts1996} is a popular Markov chain Monte Carlo algorithm (MCMC). MALA can be used to estimate integrals with respect to measures of the form
\begin{equation}
\label{eq:mu}
\pi_v(dx):=\frac{e^{v(x)}}{Z_v} \, dx,
\end{equation}
where $dx$ denotes the Lebesgue measure on $\rn$, $v:\rn\to\r$ is a smooth confining potential, and $Z_v$ is the normalising constant
$Z_v:=\int_{\rn} e^{v(x)}dx.$
It is well known that $\pi_v$ is the unique stationary measure of the \emph{Langevin diffusion}
\begin{equation}
\label{eq:LD}
dX_t=\nabla v(X_t) \, dt+\sqrt{2} \, dW_t,\qquad X_0=Z.
\end{equation}
The SDE~\eqref{eq:LD} has globally defined solutions and, regardless of the initial condition, the limit \eqref{eq:conv} holds with $\pi:=\pi_v$ for all $\pi_v$-integrable functions $f$~\cite{Roberts1996}. 
MALA proceeds by discretising $X$, adding a Metropolis accept-reject step to preserve stationarity of $\pi_v$, and simulating the resulting chain. The time averages of the simulation then converge to the desired average~\cite{Roberts1996}.

Since $\pi_v$ is the unique stationary measure of \eqref{eq:LD}, we can use our algorithm to directly compute bounds on 
$\pi_v(f)$ when both $f$ and $v$ are polynomials, circumventing any discretisation or simulation.
We illustrate this idea with the following simple Bayesian inference problem.

\begin{example}
\label{pest} 
The scalar noisy time-varying recurrence equation
\begin{equation}
\label{eq:rec}
z_k=p_1z_{k-1}+p_2\frac{z_{k-1}}{1+z_{k-1}^2}+p_3\cos(1.2(k-1))+\xi_k,
\end{equation}
is often used to benchmark parameter and state estimation algorithms~\cite{Gordon1993,Carlin1992}. For simplicity, we assume that the state $\{z_k:k=1,\dots,N\}$ is observable and we focus on the problem of estimating the parameters $p_1, p_2,$ and $p_3$. The additive noise $\{\xi_k:k=1,\dots,N\}$ is typically taken to be an i.i.d.\ sequence of normally distributed random variables. Since Gaussianity of random variables is not important in our algorithm, we instead choose $\{\xi_k:k=1,\dots,N\}$ to be an i.i.d. sequence of random variables with bimodal law
$$
\mu_\xi(dx):= \pi_{u_\xi}(dx)  \qquad \text{with} \qquad {u_\xi}(x)=3x^2-x^4.
$$
where $\pi_{u_\xi}$ is as in~\eqref{eq:mu}, see Fig.\ref{conds}a. 
Choosing parameters $p:=(p_1,p_2,p_3)=(0.5,2,1)$ and $z_0=2$, we use~\eqref{eq:rec} to generate $N$ samples $z:=\{z_1,z_2,\dots,z_N\}$, see Fig.\ref{conds}b. The inference problem is to use the generated samples $z$ and the initial condition $z_0$ to estimate the parameters $p$. 

Taking a Bayesian perspective, we first choose a prior distribution $\mu_0$ over the parameters and then we extract information from the posterior distribution $\mu_{p|z}$, see \cite{Stuart2010}. Our algorithm can be used to this end if the prior $\mu_0$ is of the form~\eqref{eq:mu} with a polynomial potential:
$$\mu_0(dx):= \pi_{u_0}(dx) \qquad \text{and  $u_0$ is a polynomial.}$$
%\frac{e^{u_0(p)}}{Z_{u_0}}dx$$
%and $u_0$ is a polynomial, 
From Bayes' formula, the posterior also takes the form~\eqref{eq:mu}: 
%$\mu_{p|z}\equiv\pi_v$ where 
%
%\begin{equation*}
%\label{eq:posterior}
$$\mu_{p|z}(dx) = \pi_v (dx) \quad \text{with}\quad
v(p)=\sum_{k=1}^N u_\xi\left(z_k-\left(p_1z_{k-1}+p_2\frac{z_{k-1}}{1+z_{k-1}^2}+p_3\cos\left(1.2(k-1)\right)\right)\right)+u_0(p).$$
%\end{equation*}
%
We can then use our algorithm to yield lower and upper bounds on the posterior means
 $\mu_{p|z}(p_1)$, $\mu_{p|z}(p_2)$ and $\mu_{p|z}(p_3)$, and upper bounds on the total variance 
$$\mathrm{var}_{p|z}=\mu_{p|z}\left((p_1-\mu_{p|z}(p_1))^2 \right)+
\mu_{p|z}\left((p_2-\mu_{p|z}(p_2))^2 \right)+
\mu_{p|z}\left((p_3-\mu_{p|z}(p_3))^2 \right).$$
of the posterior distribution $\mu_{p|z}$. For simplicity, we chose our prior $\mu_0$ to be a unit variance zero mean normal distribution, i.e., 
$u_0(p):=-\norm{p}^2/2$. The results are shown in Fig.\ref{conds}c,d. 
For $N \geq 15$ samples, we obtain small upper bounds on the total variance (two orders of magnitude smaller than the lower bounds on the posterior means). For this reason, we expect the posterior distribution to resemble a Dirac measure at the vector of posterior means, indicating that the posterior means are appropriate estimators of the parameters (as confirmed in Fig.\ref{conds}c), thus solving our inference problem.
\begin{figure}[h!]		%H: Here,T: Top, B:Bottom
	\begin{center}
	\includegraphics[width=0.9\textwidth]{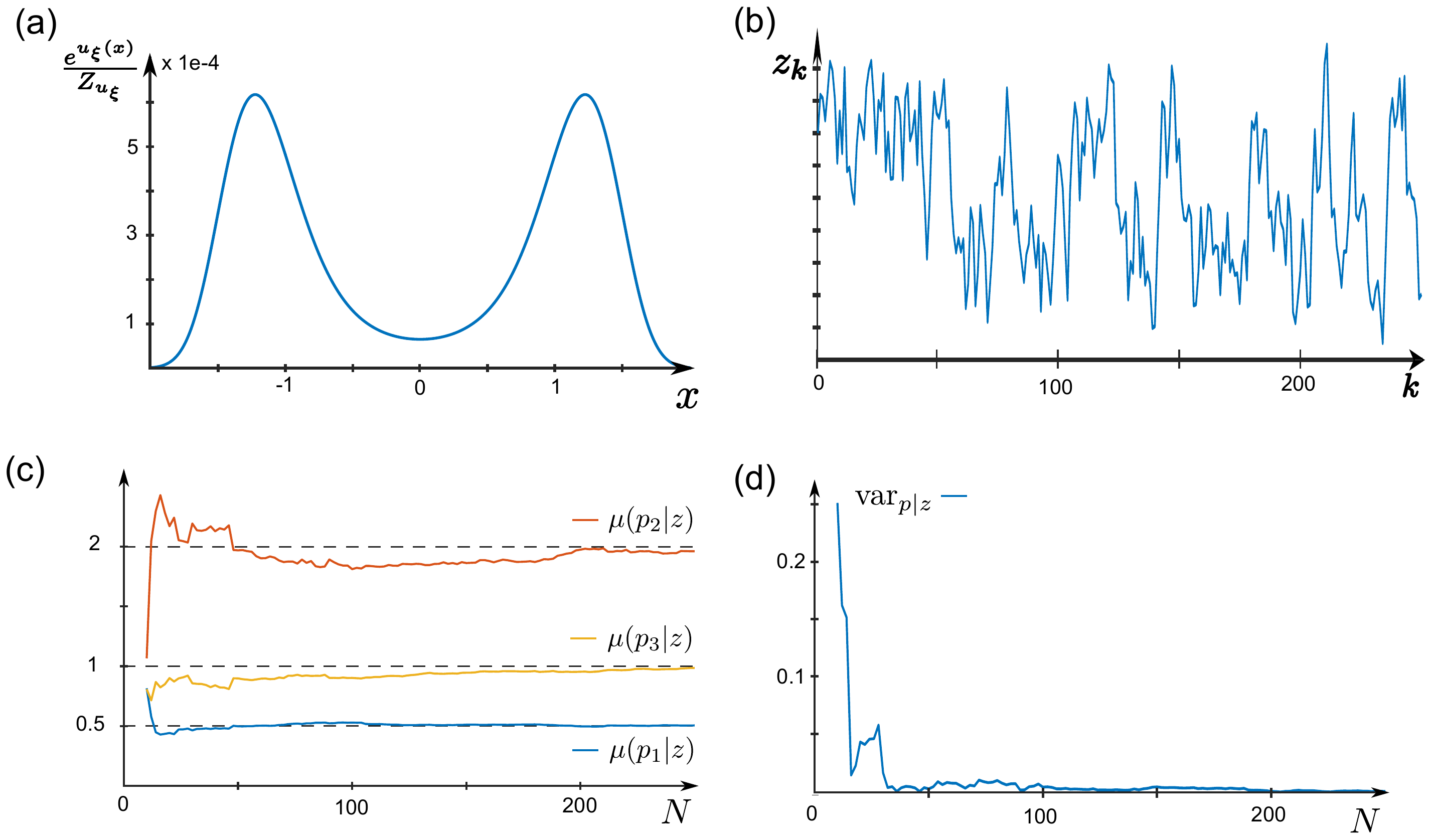} %See [OPTIONS]
	\vspace{-5pt}
	\end{center}
\caption{(a) The density of the noise $\xi$. (b) Sample path of length $N=250$ generated to estimate the parameters $p$. 
The noise was generated by inverting the cdf of $\mu_\xi$ numerically (uniform grid on $[-10,10]$ with step size of $10^{-4}$) and drawing independent samples from a uniform $[0,1]$ distribution. 
(c) Upper and lower bounds on the means $\mu_{p|z}(p_1),\mu_{p|z}(p_2), \mu_{p|z}(p_3)$ of the posterior distribution $\mu_{z|p}$ (solid lines) plotted against the number of samples $N=10,12,\dots,250$ of $z$ used to generate the posterior.  The actual values of the parameters used to generate the samples are shown with dashed lines. The upper and lower bounds were computed by solving \eqref{eq:bounds} using $d=5$ and the appropriate objective $f(x)=x_1,x_2,$ or $x_3$. The gap between the upper and lower bounds was always smaller than $10^{-2}$ and hence the upper and lower bounds are indistinguishable in the plot.
(d) Upper bounds on the total variance of the posterior distribution, $\rm{var}_{p|z}=\mu_{p|z}(p_1^2+p_2^2+p_3^2)-\mu_{p|z}(p_1)^2-\mu_{p|z}(p_2)^2-\mu_{p|z}(p_3)^2$. These were obtained by computing upper bounds on $\mu_{p|z}(p_1^2+p_2^2+p_3^2)$ (solving \eqref{eq:bounds} with $d=5$ and $f(x)=x_1^2+x_2^2+x_3^2$), and combining these bounds with the lower bounds computed for the posterior means.
In total, $847$ bounds were computed taking a total CPU time of $720$ seconds, averaging $0.85$ seconds per bound. 
 }
\label{conds}
\end{figure}
%%%%%%%%%%%%%%%%%%%%%%%%%%%%%%%
%
\end{example}
It is important to remark that Example~\ref{pest} can be solved equally well with MCMC methods---indeed, MCMC algorithms scale better than ours with the dimension of the target measure $\pi_v$. However, our alternative approach presents some attractive features: it is fast (see caption of Fig.\ref{conds}) and simple to implement; 
no tuning of the algorithm is required (e.g., choosing the discretisation step size in MALA);
and it delivers deterministic bounds on the integrals of interest instead of stochastic estimates (hence avoiding issues concerning the convergence of MCMC simulations). 
Our algorithm can also provide useful information in situations where MCMC methods face difficulties; in particular when  
the target distribution has several isolated modes. In such cases, MCMC algorithms get stuck at one of the modes and do not explore the rest of the target distribution. Our numerical observations suggest that our algorithm is also affected by the presence of isolated modes, producing a large gap between the upper and lower bounds for the desired integrals (since each bound is stuck at a different mode). The presence of such large gaps can alert the practitioner to the existence of isolated modes, something which is often not obvious for target distributions of dimension three or more. Methods designed to deal with isolated modes, such as simulated-tempering~\cite{Carlin1992}, can then be used instead of standard a Metropolis-Hastings MCMC method.
\subsection{Lyapunov exponents of linear SDEs driven by multiplicative white noise.}\label{sec42}
In practical applications involving systems of linear ordinary differential equations (ODEs), we are interested in situations where the parameters are perturbed by Gaussian white noise. In those cases, we obtain the following class of linear stochastic differential equations
\begin{equation}\label{eq:linsde}
dX(t)=AX(t) \, dt+\sum_{i=1}^mB_iX(t) \, dW_i(t),\qquad X(0)=Z,
\end{equation}
where $A\in\rnn$, $B_i\in\rnn$ and $W_i:=\{W_i(t):t\geq0\}$ are $m$ independent standard Brownian motions on~$\r$.  It is well known that \eqref{eq:linsde} has globally defined solutions. 
Furthermore, there exists a jointly continuous process $\{X^x_t:t\geq0,x\in\rn\}$ such that 
$X^Z:=\{X_t^Z:=X_t^\cdot\circ Z:t\geq0\}$ is the unique solution of \eqref{eq:linsde} (see \cite[Thrm. 21.3]{Kallenberg2001}). 

In 1967, Khas'minskii~\cite{Khasminskii1967} made the following observation. Applying It\^o's formula twice, he found that the projection of $X_t^Z$ onto the unit sphere, $\Lambda_t^Z:=X_t^Z/\norm{X_t^Z}$, satisfies the SDE
\begin{equation}\label{eq:Lam}
d\Lambda^Z(t) = u_0(\Lambda^Z(t))dt + \sum_{i=1}^mu_i(\Lambda^Z(t))dW_i(t),\qquad \Lambda^Z(0)=Z/\norm{Z},
\end{equation}
where  
$u_0(x):=Ax-\iprod{x}{Ax}x-\sum_{i=1}^m\left(\frac{1}{2}\norm{B_ix}^2x+\iprod{x}{B_ix}B_ix-\frac{3}{2}\iprod{x}{B_ix}^2x\right),$
and $u_i(x):=B_ix-\iprod{x}{B_ix}x$. 
In addition, $\Gamma_t^Z:=\ln\norm{X_t^Z}$ satisfies
\begin{equation}
\label{eq:Gam}
\Gamma^Z(t)=\ln{\norm{Z}}+\int_0^tQ(\Lambda^Z(s))ds+\sum_{i=1}^m\int_0^t\iprod{\Lambda^Z(s)}{B_i\Lambda^Z(s)}dW_i(t),
\end{equation}
where $Q(x):=\iprod{x}{Ax}+\frac{1}{2}\sum_{i=1}^m\norm{B_ix}^2-\iprod{x}{B_ix}^2.$
It is not difficult to argue \cite[Lem. 6.8]{Khasminskii2012} that
\begin{equation}
\label{eq:MLLN}
\lim_{t\to\infty}\frac{1}{t}\left(\sum_{i=1}^m\int_0^t\iprod{\Lambda^Z(s)}{B_i\Lambda^Z(s)}dW_i(t)\right)=0,\qquad\Pb\text{-almost surely}.
\end{equation}
Suppose that the generator of $\Lambda^Z:=\{\Lambda_t^Z:t\geq0\}$ is hypoelliptic on $C^2(\mathbb{S}^{n-1})$, where $\mathbb{S}^{n-1}:=\{x\in\rn:\norm{x}=1\}$ denotes the unit sphere in $\rn$. Since $\Lambda^Z$, by definition, takes values in $\mathbb{S}^{n-1}$, Theorem~\ref{drift} tells us that \eqref{eq:Lam} has at least one stationary measure, and together with \eqref{eq:Gam} and \eqref{eq:MLLN} that for $\Pb$-almost every sample path $t\mapsto X^Z(t)$ there exists a stationary measure $\pi$ of \eqref{eq:Lam} such that
\begin{equation}
\label{eq:leconv}
\lim_{t\to\infty}\frac{\Gamma^Z(t)}{t}=\lim_{t\to\infty}\frac{1}{t}\int_0^t Q\left(\Lambda^Z(s)\right) \,ds=\pi(Q).
\end{equation}
Typically, the above integral is estimated by choosing an appropriate discretisation scheme for \eqref{eq:linsde}, simulating the resulting chain, and computing the corresponding time average \cite{Talay1991}. We instead exploit the fact that $u_0,\dots u_m$ and $Q$ are all polynomials and apply our algorithm on \eqref{eq:Gam} to compute bounds for~$\pi(Q)$. In particular, for any $d\geq d_Q$ we have that: 
$$\rho^d_{Q,\mathbb{S}^{n-1}}\leq\lim_{t\to\infty}\frac{\Gamma^Z(t)}{t}\leq\eta^d_{Q,\mathbb{S}^{n-1}},
\qquad\Pb\text{-almost surely},$$
where $\rho^d_{Q,\mathbb{S}^{n-1}}$ and $\eta^d_{Q,\mathbb{S}^{n-1}}$ are as in \eqref{eq:bounds} with notation adapted to~\eqref{eq:Lam}.
Note that
\begin{align*}& \left\{\omega\in\Omega:\lim_{t\to\infty}\frac{\Gamma^Z_t(\omega)}{t}<0\right\}\subseteq\left\{\omega\in\Omega:\limsup_{t\to\infty}\norm{X^Z_t(\omega)}=0\right\},\\ & \left\{\omega\in\Omega:\lim_{t\to\infty}\frac{\Gamma^Z_t(\omega)}{t}>0\right\}\subseteq\left\{\omega\in\Omega:\liminf_{t\to\infty}\norm{X^Z_t(\omega)}=\infty\right\},\end{align*}
which implies that 
\begin{align*}
 \eta^d_{Q,\mathbb{S}^{n-1}}<0  \Rightarrow \Pb\left(\left\{\limsup_{t\to\infty}\norm{X^Z_t}=0\right\}\right)=1, \quad 
\text{for any initial condition $Z$}, \\
\rho^d_{Q,\mathbb{S}^{n-1}}>0 \Rightarrow \Pb\left(\left\{\liminf_{t\to\infty}\norm{X^Z_t}=\infty\right\}\right)=1, \quad
\text{for any initial condition $Z$},
\end{align*}
i.e., the equilibrium solution $X^0 \equiv 0$ of \eqref{eq:linsde} is almost surely asymptotically stable 
if $\eta^d_{Q,\mathbb{S}^{n-1}}<0$, and almost surely asymptotically unstable if $\rho^d_{Q,\mathbb{S}^{n-1}}>0$. 
%then the equilibrium solution of \eqref{eq:linsde} is almost surely asymptotically unstable in the sense that, for any $Z$, 
%
%$$\Pb\left(\left\{\liminf_{t\to\infty}\norm{X^Z_t}=\infty\right\}\right)=1.$$
%
Therefore, our algorithm applied to~\eqref{eq:Lam} yields a \emph{sufficient} test for the asymptotic stability 
or instability of~\eqref{eq:linsde}. The method also yields a \emph{necessary} test for asymptotic stability in the following sense:
\begin{theorem}\label{convl}Suppose that the generator of $\Lambda^Z$ is hypoelliptic on $C^2(\mathbb{S}^{n-1})$. Let
$$\lambda(Z):=\lim_{t\to\infty}\frac{\log\norm{X^Z_t}}{t},\qquad \lambda_-:=\inf_{Z}\lambda(Z),\qquad\lambda_+:=\sup_{Z}\lambda(Z),$$
where the infimum and supremum are taken over the set of initial conditions $Z$ that are Borel measurable random variables on $\rn$. For sufficiently large $d$,
$$\rho_{Q,\mathbb{S}^{n-1}}^d>-\infty,\qquad \eta_{Q,\mathbb{S}^{n-1}}^d<\infty.$$
Furthermore,
$$\lim_{d\to\infty}\rho_{Q,\mathbb{S}^{n-1}}^d=\lambda_-,\qquad \lim_{d\to\infty}\eta_{Q,\mathbb{S}^{n-1}}^d=\lambda_+.$$
\end{theorem}
\begin{proof} Let $\cal{A}$ denote the generator of $\Lambda^Z$ and not of $X^Z$. The theorem is proved only for the lower bounds $\rho_{Q,\mathbb{S}^{n-1}}^d$. The proof for the upper bounds is identical. The proof proceeds in three steps:
\begin{enumerate}
\item We show that the set of limits 
$\cal{S}:=\{\lambda(Z):Z\text{ is a Borel measurable random variable on }\rn\}$
is the same as the set 
$B_{Q,\mathbb{S}^{n-1}}:=B_{Q,\mathbb{S}^{n-1}}^1=\dots=B_{Q,\mathbb{S}^{n-1}}^\infty.$
\item We show that the equalities~\eqref{eq:eqs} fully characterise the stationary measures of \eqref{eq:Lam}, in the sense that if $\pi$ is a Borel probability measure with support contained in $\mathbb{S}^{n-1}$ such that
\begin{equation}\label{eq:eqs2}
\pi(\cal{A}x^\alpha)=0,\, \forall\alpha\in\nn,
\end{equation}
then $\pi$ is a stationary measure of \eqref{eq:Lam}. 
\item We then only have to apply Theorem 4.3 in \cite{Lasserre2009}, which shows that, for large enough $d$,  
$\rho^d_{Q,\mathbb{S}^{n-1}}>-\infty$, and
$$\lim_{d\to\infty}\rho^d_{Q,\mathbb{S}^{n-1}}=\inf \left\{\pi(Q):
\begin{array}{l}
\pi \text{ is a Borel probability measure with support contained in }\mathbb{S}^{n-1}\\ 
\text{ that satisfies $\pi(\cal{A}x^\alpha)=0,\, \forall \alpha \in \nn$}
\end{array}
\right \}.$$
Then~\eqref{eq:eqs} fully characterises the stationary measures of \eqref{eq:Lam}, which implies that 
$$\lim_{d\to\infty}\rho^d_{Q,\mathbb{S}^{n-1}}=\inf B_{Q,\mathbb{S}^{n-1}}=\inf\cal{S}=\lambda_-.$$
For the detailed proof, see Appendix \ref{appendix2}. 
\end{enumerate}
\end{proof}

\begin{example}
\label{stabuns}
We exemplify the use of our algorithm in computing Lyapunov exponents through a classic example from \cite{Khasminskii1967}. The question of whether an unstable linear system of ODEs could be stabilised by physically realisable multiplicative noise (i.e., multiplicative noise in the sense of Stratonovich) received ample attention in the 1960s. It was shown that this cannot be achieved in one-dimensional systems, and it was hypothesised that it could not be achieved in higher dimensional systems either~\cite{Leibowitz1963}. Khas'misnkii \cite{Khasminskii1967} disproved this hypothesis with the following counterexample:
\begin{equation}\label{eq:counter}\begin{matrix}
d X_1(t)=c_1 \,X_1(t)\circ d t+\sigma\left(X_1(t)\circ d W_1(t)+X_2\circ dW_2(t)\right)\\
d X_2(t)=c_2 \, X_2(t)\circ d t+\sigma\left(X_2(t)\circ d W_1(t)-X_1\circ dW_2(t)\right)
\end{matrix}
\end{equation}
where $c_1>0$, $c_2<0$, $\sigma>0$, and $\circ d$ denotes the Stratonovich differential. For this two-dimensional SDE, the projection \eqref{eq:Lam} lives in a one dimensional manifold (the unit circle), and thus by changing to polar coordinates one can find analytical expressions for its stationary measures. In this case there is a unique stationary measure and
%In the case of \eqref{eq:counter}, the projection \eqref{eq:Lam} has one unique stationary measure $\pi$ and  
%
\begin{equation}\label{eq:lyaden}\pi(Q)=\frac{\int_0^{2\pi}\left(\frac{\sigma^2}{2}+c_1\cos^2(\phi)+c_2\sin^2(\phi)\right)e^{\frac{(c_1-c_2)}{\sigma^2}\cos^2(\phi)}d\phi}{\int_0^{2\pi}e^{\frac{(c_1-c_2)}{\sigma^2}\cos^2(\phi)}d\phi}.\end{equation}
Khas'minskii then argued that one can always find a sufficiently large (in absolute value) $c_2$  and $\sigma$ so that 
$\pi(Q)<0$, i.e., so that \eqref{eq:counter} is stable. 

Instead, we can verify Khas'minskii's findings for given $c_1$ and $c_2$ and $\sigma$ by solving for the bounds \eqref{eq:bounds}. As an example, we analysed the system~\eqref{eq:counter} with $c_1=1$ and $c_2=-30$. 
From the bounds presented in Fig.\ref{lya}, we conclude that, for these parameters, \eqref{eq:counter} is stable 
for $3\lesssim\sigma\lesssim3.7$ and unstable otherwise.

\begin{figure}[h!]		%H: Here,T: Top, B:Bottom
	\begin{center}
	\includegraphics[width=0.75\textwidth]{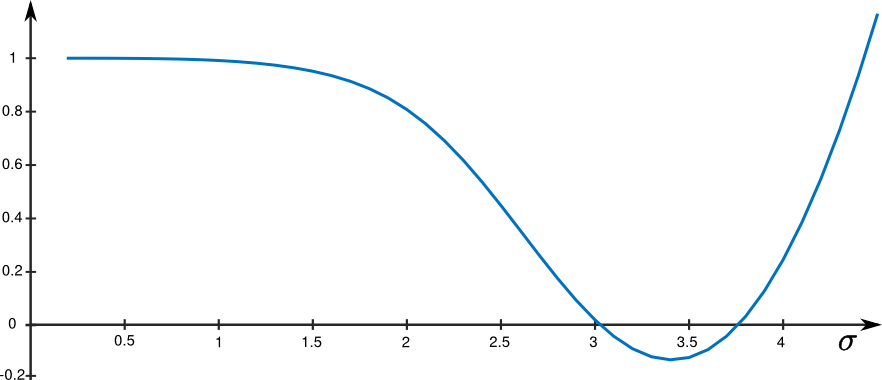} %See [OPTIONS]
	\vspace{-5pt}
	\end{center}
\caption{Upper and lower bounds on the Lyapunov exponents of the SDE~\eqref{eq:counter} with $c_1=1$ and $c_2=-30$ computed  for $\sigma=0.2,0.3,\dots,4.5$ by solving~\eqref{eq:bounds} with $d=16$. The gap between the lower and the upper bounds was always smaller than $10^{-3}$ and hence the two sets of bounds are indistinguishable in the plot. In total $88$ bounds were computed for a total CPU time of $364$ seconds, averaging $4.1$ seconds per bound.}
\label{lya}
\end{figure}
\end{example}

In Example~\ref{stabuns}, we could have evaluated~\eqref{eq:lyaden} numerically instead of computing \eqref{eq:bounds}. However,  such analytical expressions for stationary measures of \eqref{eq:Lam} are not available in higher dimensions. 
As explained in \cite{Ariaratnam1992} (see also~\cite{Weiqiu2002}): ``The direct use of Khas'minskii's method to higher dimensional systems has not met with much success because of the difficulty of studying diffusion processes occurring on surfaces of unit hyperspheres in higher dimensional Euclidean spaces". Our approach complements Khas'minskii's procedure: it is simple to implement; the number of stationary measures of \eqref{eq:Lam} is not a limitation; and the fact that the measures have support on the unit sphere is an advantage instead of a disadvantage, as Theorem~\ref{convl} shows.
\subsection{Piecewise polynomial averages, a noisy nonlinear oscillator and structural reliability problems.}\label{sec43} Our algorithm can be extended to bound stationary averages $\pi(f)$ where $f$ is a piecewise polynomial of the type
\begin{equation}
\label{eq:piece}
f=\sum_{i=1}^Nf_i \, 1_{\cal{K}_i}. 
\end{equation}
Here $f_i$ are polynomials, $1_A$ is the indicator function of set $A$, and $\cal{K}_i$ are $N$ disjoint sets in $\cal{G}$ defined by
\begin{equation}
\label{eq:K}
\cal{K}_i:=\left\{x\in\cal{G}: p^i_j(x)=0, \quad j=1,\dots,J_i,\quad q^i_k(x)\geq0, \quad  k=1,\dots,K_i \right \}.
\end{equation}
This type of extension was first discussed in~\cite{Lasserre2002}, and has been considered in \cite{Lasserre2009,Lasserre2006,Eriksson2011}. Many stationary averages of interest can be written in the above form. For instance, if $f_i:=1, \forall i$, then $\pi(f) = \pi(\cup_i\cal{K}_i)$ is the probability that event $\cup_i\cal{K}_i$ occurs.
\subsubsection{Extension of  the algorithm.} We follow here similar arguments to those presented in~\cite{Lasserre2009,Lasserre2006,Eriksson2011}. 
Let $\pi$ be a stationary measure of~\eqref{eq:sde} with support in $\mathcal{G}$ and with moments of order $d$. 
Let $\pi_i$ be the restriction of 
%
%$$\pi_0(\cdot):=\pi(\cdot\cap \mathcal{K}^c),\qquad \pi_1(\cdot):=\pi(\cdot\cap \mathcal{K}_1),\qquad\dots,\qquad \pi_N(\cdot):=\pi(\cdot\cap \mathcal{K}_N) $$
%%
%to be the restrictions of 
$\pi$ to set $\mathcal{K}_i$, and let $\pi_0$ denote the restriction to $\mathcal{K}^c:=\mathcal{G}\backslash\cup_i\mathcal{K}_i$, i.e., to the complement in $\cal{G}$ of the union of all $\mathcal{K}_i$s. Clearly,
$$\pi=\sum_{i=0}^N\pi_i.$$
Let $y^0,y^1,\dots,y^N\in\rrd$ denote the $(N+1)$ vectors of moments of $\pi_0,\pi_1,\dots,\pi_N$, respectively. 
From the definition~\eqref{eq:piece} it follows that
$$\pi(f)=\sum_{i=1}^N \pi_i \left(f_i \, 1_{\mathcal{K}_i}\right)=
\sum_{i=1}^N\pi_i(f_i)=\sum_{i=1}^N\iprod{y^i}{\pmb{f_i}},$$
assuming that $d_{f_1},\dots d_{f_N}\leq d$. The normalisation condition takes the form:
\begin{equation}
\label{eq:norm_disjoint}
\sum_{i=0}^Ny^i_0=1.
\end{equation}
Similarly to~\eqref{eq:eqs}, the proof of Lemma~\ref{diff} tells us that the stationarity of $\pi$ implies that 
\begin{equation}
\label{eq:eqs_disjoint}
\iprod{\sum_{i=0}^Ny^i}{\pmb{\cal{A}x^\alpha}}=0,\quad \forall \alpha\in\nn_{d-d_\mathcal{A}}.
\end{equation}
The vectors of moments $y^i\in\rrd$ also fulfil similar conditions to~\eqref{eq:t2} and~\eqref{eq:psd},
in particular:
\begin{equation}
\label{eq:t2_disjoint}
\iprod{y^i}{\pmb{g_jx^\alpha}}=0,\quad \forall \alpha\in \nn_{d-d_{g_j}},\quad \forall j=1,\dots,\ell, \quad \forall i=0, \ldots, N,
\end{equation}
\begin{equation}
\label{eq:t1_disjoint}
M_{s(d)}(y^i) \succeq0, \quad \forall i=0, \ldots, N.
%= \pi(m_{s(d)}\otimes m_{s(d)}) \succeq0, \quad \forall i=0, \ldots, N.
\end{equation}

Furthermore, the definition~\eqref{eq:K} of the sets $\cal{K}_i$ leads to two additional sets of conditions. 
Firstly, the vectors of moments $y^i \in \rrd$ fulfil  
%of a Borel measure with support contained in $\cal{K}_i$ with moments of order $d \geq\max_{j}d_{p^i_j}$, then:
%
\begin{equation}
\label{eq:t22_disjoint}
\iprod{y^i}{\pmb{p_j^ix^\alpha}}=0,\qquad \forall\alpha\in \nn_{d-d_{p^i_j}},\qquad\forall j=1,\dots,J_i, \quad \forall i=1, \ldots, N.
\end{equation}  
Secondly, using similar arguments as those leading to~\eqref{eq:psd}, we obtain additional moment conditions. Let $q$ be a polynomial that is non-negative on $\cal{K}_i$ and $p$ is any polynomial of degree $s(d-d_q)$. Then we have
$$q(x)(p(x))^2=\iprod{(q(x) \,m_{s(d-d_q)}(x)\otimes m_{s(d-d_q)}(x))\pmb{p}}{\pmb{p}}.$$
Since $\pi_i$ has support in $\cal{K}_i$
\begin{equation}
\label{eq:locapsd}
0\leq \pi_i(q \,p^2)=\iprod{\pi_i(q \, m_{s(d-d_q)}\otimes m_{s(d-d_q)})\pmb{p}}{\pmb{p}}
=\iprod{M_{s(d-d_q)}(\theta_qy^i) \, \pmb{p}}{\pmb{p}}
\end{equation}
where $\theta_q:\rrd\to\r^{r(d-d_q)}$ is the shift operator
$(\theta_q y)_\alpha=\sum_{\beta\in\nn_{d_q}}\pmb{q}_\beta y_{\alpha+\beta},$  and
the \emph{localising matrix} 
$$M_{s(d-d_q)}(\theta_qy^i):=\pi_i(q \, m_{s(d-d_q)}\otimes m_{s(d-d_q)})\in\r^{s(d-d_q)\times s(d-d_q)}$$ 
is the element-wise integration of the matrix $q\,m_{s(d-d_q)}\otimes m_{s(d-d_q)}$. Since \eqref{eq:locapsd} holds for all $p\in\pol_{s(d-d_q)}$, the localising matrix $M_{s(d-d_q)}(\theta_qy)$ is positive semidefinite. 
From the definition~\eqref{eq:K}, $q^i_k$ is nonnegative on $\cal{K}_i$, hence we obtain our additional set of moment conditions:
% for any $k= 1,\dots,K^i$. Thus we have that 
%
\begin{equation}
\label{eq:loca_disjoint}
M_{s(d-d_{q_k^i})}(\theta_{q_k^i}y^i)\succeq 0,\qquad\forall k=1,\dots,K_i, \quad \forall i=1, \ldots, N.
\end{equation}
Together,~\eqref{eq:norm_disjoint}% \eqref{eq:eqs_disjoint}, \eqref{eq:t2_disjoint}, \eqref{eq:t1_disjoint},
--\eqref{eq:t22_disjoint} and~\eqref{eq:loca_disjoint} provide computationally tractable necessary conditions satisfied by the moments of any measure with support on $\mathcal{K}_i$. Thus we have the following outer approximation of $B^d_{f,\cal{G}}$:
%\begin{equation}
%\label{eq:Cd_disjoint}
%%B^d_{f,\cal{G}}\subseteq 
%C^d_{f,\cal{G}}:= \left\{
%\sum_{i=1}^N \iprod{y^i}{\pmb{f_i}}:y^0,\dots,y^N\in \rrd,
%%
%\begin{array}{ll}
%\sum_{i=0}^Ny^i_0=1 &
%\eqref{eq:norm_disjoint} \\
%\iprod{\sum_{i=0}^Ny^i}{\pmb{\cal{A}x^\alpha}}=0,\, \alpha\in\nn_{d-d_\mathcal{A}} 
%& \eqref{eq:eqs_disjoint}\\
%\iprod{y^i}{\pmb{g_jx^\alpha}}=0,\, \alpha\in \nn_{d-d_{g_j}},\, j=1,\dots,\ell, \,  i=0, \ldots, N
%& \eqref{eq:t2_disjoint} \\
%M_{s(d)}(y^i) \succeq0, \, i=0, \ldots, N
%& \eqref{eq:t1_disjoint} \\
%\iprod{y^i}{\pmb{p_j^ix^\alpha}}=0,\, \alpha\in \nn_{d-d_{p^i_j}},\, j=1,\dots,J_i, \, i=1, \ldots, N
%& \eqref{eq:t22_disjoint} \\
%M_{s(d-d_{q_k^i})}(\theta_{q_k^i}y)\succeq 0,\, k=1,\dots,K_i, \, i=1, \ldots, N
%& \eqref{eq:loca_disjoint}
%\end{array}
%\right \}.
%\end{equation}
%
\begin{equation}\label{eq:Cd_disjoint}
%B^d_{f,\cal{G}}\subseteq 
C^d_{f,\cal{G}}:= \left\{
\sum_{i=1}^N \iprod{y^i}{\pmb{f_i}}:y^0, \dots,y^N\in \rrd 
\mathrel{}\middle|\mathrel{}
\eqref{eq:norm_disjoint}, \eqref{eq:eqs_disjoint}, \eqref{eq:t2_disjoint}, \eqref{eq:t1_disjoint},
\eqref{eq:t22_disjoint}, \eqref{eq:loca_disjoint} \text{ are fulfilled}
 \right \}.
\end{equation}
The extended algorithm obtains bounds on $B^d_{f,\cal{G}}\subseteq C^d_{f,\cal{G}}$ by finding the infimum and supremum of $C^d_{f,\cal{G}}$; that is, by solving two SDPs with $(N+1)$ times as many variables as in our original algorithm.
\subsection*{Remark (Support of $\pi_0$).} Notice that $C^d_{f,\cal{G}}$ does not include conditions on $y^0$ related to the support of $\pi_0$. The reason why is that, in general, $\mathcal{K}^c$ has no simple description of the form \eqref{eq:K}. If such a description exists, extra constraints can easily be appended. 
If such constraints are lacking, there is an unfortunate consequence: $C^d_{f,\cal{G}}$ always contains the origin, since 
$(y^0,y^1,\dots,y^N):=(y,0,\dots,0)$ satisfies all the constraints in~\eqref{eq:Cd_disjoint}. 
Consequently, our extended algorithm only yields informative bounds 
if $f$ is nonnegative (or nonpositive) and, in this case, only the upper (resp. lower) bound will be informative. 
However, for certain purposes this may be sufficient, as the following example demonstrates.
\begin{example} 
\label{exp:reliability}
We consider the analysis of a noisy nonlinear oscillator in relation to reliability problems of structural mechanics. Structures (e.g., buildings, bridges) perturbed by random forces (e.g. waves, earthquakes) are often modelled with the stationary response of a non-linear oscillator driven by Gaussian white noise~\cite{Dunne1997,Schueller1997}. A typical such model used in the literature is the Duffing oscillator
\begin{equation}
\label{eq:osci}
\ddot{Y}_t+\dot{Y}_t+Y_t+\frac{1}{2}Y_t^3=\sqrt{2}dW_t,
\end{equation}
where $Y_t \in \r$ describes the deviation of the structure from its standing point at time $t$, and $W:=\{W_t:t\geq0\}$ is a standard Brownian motion. 
Structural reliability examines whether the structure will bend past a critical value 
of deviation, and how often this event should be expected to occur \cite{Schueller1997}. 
First passage times (i.e., the amount of time it takes for the structure to bend past the critical deviation) 
and extreme-value probabilities (i.e., how likely the building is to bend past the critical deviation in a given interval of time) 
are typically investigated. Passage times feature more prominently in the structural mechanics literature since they are more amenable to analysis~\cite{Dunne1997}. Here, we use our extended algorithm to find \emph{upper bounds} on extreme-value probabilities and on the average fraction of time the structure spends bent beyond the critical deviation.

It is well known~\cite[\S 3]{Mattingly2002} that \eqref{eq:osci} has a unique stationary measure 
\begin{equation}\label{eq:dufden}\pi(dy\times d\dot{y})=e^{- \left(\frac{\dot{y}^2}{2}+\frac{y^2}{2}+\frac{y^4}{8}\right)}dy\times d\dot{y},\end{equation}
where $dy\times d\dot{y}$ denotes the Lebesgue measure in $\r^2$. Throughout this section we assume that the process is at stationarity, i.e., 
$(Y_t,\dot{Y}_t)\sim\pi,$ $\forall t\geq0$.
We begin by re-writing \eqref{eq:osci} in It\^o form
$$dX(t):=\begin{bmatrix}dX_1(t)\\ dX_2(t)\end{bmatrix}=\begin{bmatrix}X_1(t)\\ -X_2(t)-X_1(t)-\frac{1}{2}X^3_1(t)\end{bmatrix}dt+\begin{bmatrix}0\\ \sqrt{2}\end{bmatrix}dW_t,$$
where $X_1:=Y$ and $X_2:=\dot{Y}$. By Birkhoff's Ergodic Theorem \cite[\S 3]{Bellet2006}, the average fraction of time that the building spends bent beyond the critical deviation $u$ converges to 
\begin{equation}
\label{eq:Fu}
F_u:=\pi(1_{\{x\in\r^2:x_1\geq u\}}).
\end{equation}
The extreme-value probability is defined as
\begin{equation}
\label{eq:ext_val_prob}
P_u:=\Pb\left(\left\{\sup_{s\in[0,T]}X_1(s)>u\right\}\right),
\end{equation}
where $[0,T]$ is a given time interval of interest. For sufficiently high $u$, it is shown in \cite{Leadbetter1983} 
that the up-crossing events become independent; hence, the number of up-crossings in the interval $[0,T]$ has a Poisson distribution with mean $v_u \, T$, where $v_u$ is the mean threshold crossing rate of $u$. For sufficiently regular stationary processes, the mean up-crossing rate can be obtained from Rice's formula~\cite{Naess1984,Soong1973}
\begin{equation}
\label{eq:upc}
v_u:=\pi \left(x_2 \, 1_{\{x\in\r^2:x_1=u,x_2\geq0\}} \right),
\end{equation}
whence it follows that 
\begin{equation}
\label{eq:pois}
P_u\approx 1-e^{-v_uT}.
\end{equation}
In the computations below, we characterise this regime and consider crossings over high deviations $u$ (at least three times larger than the standard deviation of $\pi$); hence we assume that \eqref{eq:pois} holds exactly. From \eqref{eq:Fu} and \eqref{eq:upc}, it is clear that our extended algorithm can provide upper bounds on $v_u$ and $F_u$.  These bounds are then used with equation \eqref{eq:pois} to bound $P_u$.
%
%Now choose $h(x):=x_1-u$. Clearly, $F_u$ is just $y_{(0,0)}$, the mass of $\nu$. So, 
%%
%$$F_u\in D_u^d:=\left\{y_{(0,0)}:\exists y,z\in \rrd,\begin{array}{l}y_0+z_0=1,\\\iprod{y+z}{\pmb{\cal{A}x^\alpha}}=0, \quad\forall \alpha\in\n^2_{d-3}\\ M_{s(d)}(y)\succeq0,\quad M_{s(d)}(z)\succeq0,\quad M_{s(d-1)}(\theta_{x_1-u}y)\succeq0  \end{array}\right\}.$$
%%
%Similarly, we have that
%%
%$$v_u\in E_u^d:=\left\{y_{(0,1)}:\exists y,z\in \rrd,\begin{array}{l}y_0+z_0=1,\\\iprod{y+z}{\pmb{\cal{A}x^\alpha}}=0,\quad \forall \alpha\in\n^2_{d-3}\\ (y_\alpha)_1=u^{\alpha_1}, \qquad \forall \alpha\in\n^2_{d-3} \\ M_{s(d)}(y)\succeq0,\quad M_{s(d)}(z)\succeq0,\quad M_{s(d-1)}(\theta_{x_2}y)\succeq0  \end{array}\right\}.$$
%

The standard deviation of $\pi$ is $\sigma \approx 0.761$, as computed directly from~\eqref{eq:dufden} or 
using our unmodified algorithm. %to obtain tight bounds on the second moment of $\pi$.
Following~\cite{Dunne1997}, we considered a time interval $T=100$ and critical deviations $u$ varying from $3 \sigma$ to 
$5 \sigma$. The largest SDP successfully solved by SDPT3 was $d=14$, and we computed $14$ bounds in a CPU time of $403$ seconds, averaging $29$ seconds per bound. 
The upper bounds computed with our algorithm are shown with the exact values computed using~\eqref{eq:dufden} in Tables \ref{tab3} and \ref{tab4}. Although the upper bounds are orders of magnitude greater than the true $P_u$ and $F_u$, the bounds could be useful for practical purposes. For instance, our bounds state that the probability of the structure bending further than five standard deviations at any point over the interval of time $[0,100]$ is less than $0.011\%$, and that the structure will spend less than $2$ millionth's of a percent of that interval bent beyond this deviation. 

This example was chosen so that exact values of $F_u$ and $P_u$ could be computed directly from \eqref{eq:dufden},
so as to evaluate the quality of the bounds. For most oscillator models, no such analytical expressions are available, and it is often not even clear how many stationary measures exist. Our method applies equally to these other oscillator models.
\begin{table}[!ht]
\begin{center}
\resizebox{\textwidth}{!}{
\begin{tabular}{c|ccccccc}%\hline
$u/\sigma$&$3$&$3\tfrac{1}{3}$&$3\tfrac{2}{3}$&$4$&$4\tfrac{1}{3}$&$4\tfrac{2}{3}$&$5$\\\hline
Upper bound& $6.326\times10^{-2}$   & $1.673\times10^{-2}$   &    $4.898\times10^{-3}$  &$1.622\times10^{-3}$  &$5.905\times10^{-4}$    & $2.469\times10^{-4}$   & $1.059\times10^{-4}$ \\%\hline
Exact value& $4.581\times10^{-2}$   & $4.280\times10^{-3}$   &    $1.974\times10^{-4}$  &$4.015\times10^{-6}$  &$3.126\times10^{-8}$    & $7.978\times10^{-11}$   & $5.639\times10^{-14}$ 
\end{tabular}}
\caption{Computed upper bounds on $P_u$ and exact values from~\eqref{eq:pois}.}\vspace{-20pt}
\label{tab3}
\end{center}
\end{table}
\begin{table}[!ht]
\begin{center}
\resizebox{\textwidth}{!}{
\begin{tabular}{c|ccccccc}%\hline
$u/\sigma$&$3$&$3\tfrac{1}{3}$&$3\tfrac{2}{3}$&$4$&$4\tfrac{1}{3}$&$4\tfrac{2}{3}$&$5$\\\hline
Upper bound& $4.804\times10^{-4}$   & $3.272\times10^{-5}$   &    $3.781\times10^{-6}$  &$8.814\times10^{-7}$  &$1.799\times10^{-7}$    & $5.903\times10^{-8}$   & $1.7954\times10^{-8}$ \\%\hline
Exact value & $1.280\times10^{-4}$   & $9.267\times10^{-6}$   &    $3.409\times10^{-7}$  &$5.601\times10^{-9}$  &$3.560\times10^{-9}$    & $7.493\times10^{-14}$   & $4.413\times10^{-17}$ 
\end{tabular}}
\caption{Computed upper bounds on $F_u$ and exact values from~\eqref{eq:ext_val_prob}.}\vspace{-20pt}
\label{tab4}
\end{center}
\end{table}
\end{example}
\section{Concluding remarks.}\label{disc} In this paper, we have introduced an algorithm based on semidefinite programming that yields upper and lower bounds on stationary averages of SDEs with polynomial drift and diffusion coefficients. The motivation behind our work is the study of long-term behaviour of such SDEs. As explained in the introduction, additional work is required to link the bounds obtained by our algorithm with this long-term behaviour. Typically, a drift condition must be verified by finding an appropriate Lyapunov function~\cite{Meyn1993b}. For polynomial drift vectors and diffusion matrices, one can also employ semidefinite programming to search for these Lyapunov functions (see, \emph{sum of squares} programming approaches~\cite{Parrilo2000,Parrilo2003}). 
In many respects, these approaches are dual to the method we describe in this paper, see \cite{Laurent2009,Lasserre2009,Glynn2008,Kashima2011a,Kashima2013} for more on the connections. 

Our algorithm is also applicable to SDEs whose diffusion coefficients $\sigma$ are not polynomial, but whose diffusion matrix $a:=\sigma\sigma^T$ is polynomial (e.g., the Cox-Ingersoll-Ross interest rate model in financial mathematics). 
We have concentrated on polynomial diffusion coefficients for simplicity, in order to guarantee uniqueness of solutions. Furthermore, our algorithm can be extended to SDEs with \emph{rational} drift vector and diffusion matrix---one must just carefully choose polynomials $h$ such that $\cal{A}h$ is still a polynomial.

Our choice of moment constraints was motivated by the convenience of use of the modelling package GloptiPoly 3. However, there is a wide selection of moment conditions, some of which lead to easier conic programs (e.g., linear programs or second-order cone programs)~\cite{Lasserre2009,Lasserre2004}. 
We also restricted ourselves to stationary measures with supports contained in algebraic varieties. We did this to simplify the exposition and because the applications chosen did not require more generality. However, from Section \ref{sec43} it is clear that a similar algorithm can be constructed for measures with supports in so called 
\emph{basic semialgebraic sets} of the form~\eqref{eq:K}~\cite{Lasserre2009}. 
Such an approach could be advantageous for Example~\ref{x3}---using Stroock Varadhan's Support Theorem it is not difficult to deduce that any stationary measure of those SDEs must have support on the nonnegative semiaxis $[0,\infty)$.

Lastly, a practical issue relating to numerical aspects of our algorithm, and of GMAs in general.  
In contrast with other approaches, our algorithm runs in polynomial time---its computational complexity follows from that of primal-dual interior point algorithms~\cite{Labit2002}. However, in our experience, the applicability of the algorithm is affected by certain numerical issues. Specifically, the SDPs that arise from moment problems can be badly conditioned, causing the solvers to perform badly for medium to large problems. We believe that this is a consequence of using raw moments as the basis of the space of moment vectors, as is done in most GMAs. Indeed, the order of magnitude of the moments of a distribution varies rapidly (e.g., see~\eqref{eq:invagammon} in Example \ref{GBM}). 
Since the feasible set $\mathcal{O}^d_\mathcal{G}$ over which we optimise contains such a vector of moments and these sets are often compact~\cite{Lasserre2009}, we expect large discrepancies in the order of magnitude of the entries of our feasible points $y$ and of the moment matrix $M_d(y)$. This can lead to a bad condition number of $M_d(y)$ and consequently to poor performance of the solver. Improvements could be achieved by using an orthonormal basis with respect to the measures of interest, but this is not easy in practice; not only do we usually have little \textit{a priori} information about the measures to guide our choice of basis, but the necessary modifications of the algorithmic implementation are substantial and the package GloptiPoly 3 could not be used in its current form. In our experience, a simple way to mitigate the bad conditioning is to scale the entries of the vectors $y$ by $\tilde{y}_\alpha:=y_\alpha/z^\alpha$, where $z \in \rn_+$. This is equivalent to scaling by the moments of a Dirac measure at $z$, so $z$ should be chosen so that the entry $z_i$ is close to the absolute value of the $i^{th}$ component of the mean of the measure of interest. A similar scaling was employed in \cite{Helmes2001}. It is then straightforward to show that $\tilde{y}$ satisfies the same semidefinite inequalities as $y$, and all that is left to do is to rewrite the equality constraints in terms of the rescaled variables $\tilde{y}$.
\subsection*{Acknowledgements:} We thank Nikolas Kantas and Philipp Thomas for many helpful discussions. 
%J.K. gratefully acknowledges support through a BBSRC PhD studentship (BB/F017510/1). 
%G-B.S. acknowledges support through an EPSRC Fellowship (EP/M002187/1). 
%M.B. acknowledges support from EPSRC grants EP/I032223/1, EP/I017267/1, and EP/N014529/1.
%
\bibliographystyle{siamplain} 
\footnotesize
\bibliography{library}
\normalsize
\appendix\section{Proof of Lemma \ref{diff}}\label{appendix1}
\begin{proof}[Lemma \ref{diff}] Let $X$ be a solution of \eqref{eq:sde} whose initial condition has law $\pi$. Suppose that $D^h$ (defined in \eqref{eq:mart}) is a martingale and not just a local martingale. Then
$$\Ebb{D^h_t}=\Ebb{\Ebb{\left.D^h_t\right|\cal{F}_0}}=\Ebb{D^h_0}=0,$$
for any $t>0$. Since $\pi$ is a stationary measure of \eqref{eq:sde}, $\Ebb{h(X_t)}=\Ebb{h(X_0)}=\pi(h)$ and so the above implies 
$$\Ebb{\int_0^t\cal{A}h(X_s)ds}=0.$$
Using Tonelli's Theorem and stationarity we have
$$\Ebb{\int_0^t|\cal{A}h(X_s)|ds}=\int_0^t\Ebb{|\cal{A}h(X_s)|}ds=\int_0^t\pi(|\cal{A}h|)ds=t\pi(|\cal{A}h|).$$
Since $h$ is of degree less or equal than $d-\max\{d_{b_i},d_{a_{ij}}\}$, our assumptions on $\pi$ imply that $\pi(|\cal{A}h|)<\infty$. Thus, we have that $\cal{A}h(X_s)$ is integrable with respect to $\Pb\times \lambda_t$, where $\lambda_t$ denotes the Lebesgue measure on $[0,t]$. Choosing any $t>0$ and applying Fubini's Theorem we obtain
$$t\pi(\cal{A}h)=\int_0^t\Ebb{\cal{A}h(X_s)}ds=\Ebb{\int_0^t\cal{A}h(X_s)ds}=0$$
from which it follows that $\pi(\cal{A}h)=0$.

We now need to argue that $D^h$ is indeed a martingale. To show that $D^{h}$ is a martingale it suffices to show that $\Ebb{\left. D^h_t\right| \cal{F}_s}= D^h_s$ for all $t\geq s$. Equivalently that $\Ebb{1_AD^h_t}=\Ebb{1_AD^h_s}$ for any $A\in\cal{F}_s$ and $t\geq s$. We do this by finding a sequence of martingales $\{D^{h_m}:m\in\z\}$ such that for every $t\geq0$, $\{D^{h_m}_t:m\in\z\}$ is dominated by a $\Pb$-integrable random variable and such that $\{D^{h_m}_t:m\in\z\}$ converges almost surely to $D^h_t$. With such a sequence at hand, we can use dominated convergence and the martingale property of $D^{h_m}$ to establish the desired result:
$$\Ebb{1_A D^{h}_t}=\lim_{m\to\infty}\Ebb{1_A D^{h_m}_t}=\lim_{m\to\infty}\Ebb{1_A D^{h_m}_s}=\Ebb{1_A D^{h}_s}.$$
Thus, all that remains is to construct the sequence $\{D^{h_m}:m\in\z\}$. We do so by using the fact that if $g\in C^2(\rn)$ is compactly supported, then $D^g$ is a martingale, see \cite[\S V]{Rogers2000b}. For any natural number $m$ let
$$h_m(x):=\phi(x_1/m)\phi(x_2/m)\dots \phi(x_n/m)h(x)$$
where $\phi$ is the smooth compactly supported function
$$\phi(y):=\left\{\begin{array}{c l}\exp\left(-\frac{y^2}{1-y^2}\right)&\text{ if }|y|<1\\0&\text{ otherwise}\end{array}\right. .$$
By definition of $h_m$, we have that $D^{h_m}_t$ tends almost surely to $D^h_t$ and so we just need to show that $D^{h_m}_t$ is dominated by a $\Pb$-integrable random variable. 

Since $\partial_z\phi(z/m)|_{z=y}=\partial_x\phi(x)/m|_{x=y/m}$, and $\partial_{zz}\phi(z/m)|_{z=y}=\partial_{xx}\phi(x)/m^2 |_{x=y/m}$, and $\phi,\partial\phi,$ and $\partial^2\phi$ are all bounded (since they are continuous functions non-zero only on the compact set $[-1,1]$) we have that
$$|\cal{A} h_m|\leq\beta\left(\sum_{i=1}^n (|hb_i|+|\partial_i hb_i|+\sum_{i,j=1}^n |ha_{ij}|+|\partial_i ha_{ij}|+|\partial_j ha_{ij}| +|\partial_{ij} ha_{ij}| \right)$$
where $\beta$ is a constant that depends on the maximums of $\phi,\partial\phi,$ and $\partial^2\phi$. Since all the polynomials in the right-hand side are of degree $d$ or less, the right-hand side is $\pi$-integrable. Thus, the sequence of random variables $\{D_t^{h_m}\}_{m\in\z_+}$ is dominated by
$$|h(X_t)|+|h(X_0)|+\beta\int_0^t\sum_{i=1}^n (|h(X_s)b_i(X_s)|+|\partial_i h(X_s)b_i(X_s)|+\sum_{i,j=1}^n |h(X_s)a_{ij}(X_s)|+|\partial_i h(X_s)a_{ij}(X_s)|$$
$$+|\partial_j h(X_s)a_{ij}(X_s)| +|\partial_{ij}h(X_s) a_{ij}(X_s)|ds.$$
Using Tonelli's Theorem, and stationarity of $X$ as we did before, we have that the above is $\Pb$-integrable and the lemma follows.
\end{proof}
\section{Proof of Theorem \ref{convl}}\label{appendix2}
\begin{proof}[Theorem \ref{convl}] As explained in the main text, all we have to do is carry out Steps 1 and 2.

\textbf{Step 1:} Equation \eqref{eq:leconv} tells us that $\cal{S}\subseteq B_{Q,\mathbb{S}^{n-1}}$. Conversely, suppose that $\pi$ is a stationary measure of \eqref{eq:Lam} with support contained in $\mathbb{S}^{n-1}$ and choose an initial condition $Z$ with law $\pi$. Since $\Lambda^Z(0)=Z/\norm{Z}=Z$, sationarity of $\pi$ and the Markov property of $\Lambda^Z$ implies that $\Lambda^Z$ is a stationary process with one-dimensional law $\pi$. Thus 
$$\lambda(Z)=\lim_{t\to\infty}\frac{1}{t}\int_0^tQ(\Lambda^Z(s))ds=\pi(Q),\qquad \Pb\text{-almost surely},$$
where the first equality follows from \eqref{eq:Gam} and \eqref{eq:MLLN}, and the second is a consequence of Birkhoff's Ergodic Theorem (for instance, see  \cite[\S 3]{Bellet2006}). So $B_{Q,\mathbb{S}^{n-1}}\subseteq \cal{S}$ and we have the desired $\cal{S}=B_{Q,\mathbb{S}^{n-1}}$.

\textbf{Step 2:} Suppose that $\pi$ is a Borel probability measure with support contained in $\mathbb{S}^{n-1}$ that satisfies \eqref{eq:eqs2}. We have to argue that $\pi$ is a stationary measure of \eqref{eq:Lam}, that is that if $Z$ has law $\pi$, then $\Pbb{\Lambda_t^Z\in A}=\pi(A)$ for each $t\geq0$ and each Borel measurable set $A\subseteq \mathbb{S}^{n-1}$. By approximating indicator functions with smooth functions, it is enough to argue that $\Ebb{h(\Lambda_t^Z)}=\pi(h)$ for each $t\geq0$ and each smooth function $h:\rn\to\r$. 

We first show that $\pi(\cal{A}h)=0$ for all smooth functions $h$. By linearity, \eqref{eq:eqs2} implies that $\pi(\cal{A}p)=0$ for all polynomials $p$. Using Weierstrass' Approximation Theorem it is straightforward to argue that for any smooth $h$ and $\varepsilon>0$ there exists a polynomial $p$ such that
\begin{equation}\label{eq:approx}\max_{x\in\mathbb{S}^{n-1}}\left(\mmag{h(x)-p(x)}+\sum_{i=1}^n\mmag{\partial_i h(x)-\partial_i p(x)}+\sum_{i,j=1}^n\mmag{\partial_{ij} h(x)-\partial_{ij} p(x)}\right)\leq \varepsilon.\end{equation} 
see Lemma \ref{approx2} below. Using the above we have that 
$$\mmag{\cal{A}h(x)-\cal{A}p(x)}\leq\varepsilon\left(\max_{x\in\mathbb{S}^{n-1}}\left(\sum_{i=1}^n\mmag{b_i(x)}+\sum_{i,j=1}^n\mmag{a_{ij}(x)}\right)\right),\qquad\forall x\in\mathbb{S}^{n-1}.$$
%
%$$\cal{A}h\leq\cal{A}p+\mmag{\cal{A}(h-p)}\leq \cal{A}p+\varepsilon\left(\max_{x\in\mathbb{S}^{n-1}}\left(\sum_{i=1}^n\mmag{b_i(x)}+\sum_{i,j=1}^n\mmag{a_{ij}(x)}\right)\right),$$
%
%$$\cal{A}h\geq\cal{A}p-\mmag{\cal{A}(h-p)}\geq \cal{A}p-\varepsilon\left(\max_{x\in\mathbb{S}^{n-1}}\left(\sum_{i=1}^n\mmag{b_i(x)}+\sum_{i,j=1}^n\mmag{a_{ij}(x)}\right)\right).$$
%
Since $\pi(\cal{A}p)=0$, $\pi$ has finite mass, and $\varepsilon$ was arbitrary, it follows that $\pi(\cal{A}h)=0$.

Let us return to arguing that $\Ebb{h(\Lambda^Z_t)}=\pi(h)$, where $Z$ is any initial condition with law $\pi$. Since the drift vector and diffusion coefficients of \eqref{eq:linsde} are linear, we can find a modification of $\{X^x_t:t\geq0,x\in\rn\}$ such that for every $t\geq0$, the map $\rn\ni x\mapsto X^x_t\in\rn$ is smooth, $\Pb$-almost surely, see Proposition 2.2 in \cite[\S5]{Watanabe1989}. Consequently, for every $t\geq0$, the map $x\mapsto \Lambda^x_t$ is also smooth, $\Pb$-almost surely. Now choose any smooth function $h:\rn\to\r$. It\^o's formula tells us that
$$h(\Lambda_t^Z)=h(\Lambda_0^Z)+\int_0^t(\cal{A}h)(\Lambda_s^Z)ds+\int_0^t\iprod{U(\Lambda_s^Z)\nabla h(\Lambda_s^Z)}{dW_s},\qquad \Pb\text{-almost surely},$$
where $U$ is the matrix with columns $u_1,\dots,u_m$.
By definition, the paths of $\Lambda^Z$ are contained in a compact set, thus applying the Dominated Convergence Theorem we have that 
\begin{equation}\label{eq:mart2}M_t:=\int_0^t\iprod{U(\Lambda_s^Z)\nabla h(\Lambda_s^Z)}{dW_s}\end{equation}
is not just a local martingale, but a martingale. Thus
$$\Ebb{h(\Lambda_t^Z)}=\pi(h)+\Ebb{\int_0^t(\cal{A}h)(\Lambda_s^Z)ds}.$$
Because for any $x\in\mathbb{S}^{n-1}$ the paths of $\Lambda^x$ are contained in $\mathbb{S}^{n-1}$, and because both the drift vector and the diffusion matrix are continuous functions, the map $[0,\infty)\times\mathbb{S}^{n-1}\ni (t,x)\mapsto(\cal{A}h)(\Lambda_t^x)\in \r$ is bounded. Thus we can apply Fubini's Theorem  to obtain that
$$\Ebb{\int_0^t(\cal{A}h)(\Lambda_s^Z)ds}=\int_0^t\Ebb{(\cal{A}h)(\Lambda_s^Z)}ds.$$
So it is clearly enough to argue that $\Ebb{(\cal{A}h)(\Lambda_t^Z)}=0$ for each $t\geq0$. Deconvolving the expectation we have that
$$\Ebb{(\cal{A}h)(\Lambda_t^Z)}=\int\Ebb{(\cal{A}h)(\Lambda_t^x)}\pi(dx).$$
Let $u_t(x):=\Ebb{h(\Lambda_t^x)}$. Since $\mathbb{S}^{n-1}\ni x\mapsto h(\Lambda_t^x)\in\r$ is bounded, we can differentiate under the expectation sign. Thus, for each $t\geq0$, smoothness of $h$ and of $x\mapsto\Lambda_t^x$ implies that $x\mapsto u_t(x)$ is a smooth function too. If we can show that $\Ebb{(\cal{A}h)(\Lambda_t^x)}=(\cal{A}u_t)(x)$ for each $t\geq0$, then we are done since $\pi(\cal{A}u_t)=0$. Arguing this fact is routine, see Lemma \ref{FP} below.
\end{proof}
We find it convenient in the proof of the following lemma to write $\partial_\alpha f$ as a shorthand for
$$\frac{\partial^{\alpha_1}\partial^{\alpha_2}\dots\partial^{\alpha_n}f}{\partial x_1^{\alpha_1}\partial x_2^{\alpha_2}\dots\partial x_n^{\alpha_n}}$$
where $\alpha$ is any tuple in $\nn$.
\begin{lemma}\label{approx2}Suppose that $\cal{M}$ is a compact smooth embedded submanifold of $\rn$, that $h:\cal{G}\to\r$ is smooth and choose $\varepsilon>0$. Then, there exists a polynomial $p:\rn\to\r$ such that \eqref{eq:approx} holds (with $\mathbb{S}^{n-1}$ replaced by $\mathcal{M}$).
\end{lemma}
\begin{proof}
Let $\tilde{h}$ be any compactly supported smooth extension of $h$ on $\rn$. Choose $l$ such that the support of $\tilde{h}$ is contained in the hypercube $[-l,l]^n$. For any real-valued continuous function $f$ on $\rn$, define
$$G_if(x):=\int_{-l}^{x_i}f(x_1,\dots,x_{i-1},y,x_{i+1},\dots,x_n)dy,\qquad i=1,\dots,n$$
and
$$G_\alpha f(x):=\underbrace{G_1G_1\dots G_1}_{\alpha_1\text{ times}}\underbrace{G_2G_2\dots G_2}_{\alpha_2\text{ times}}\dots\underbrace{G_nG_n\dots G_n}_{\alpha_n\text{ times}}f(x),\qquad \alpha\in\nn.$$
Now choose any $\varepsilon>0$. By the Weierstrass Approximation Theorem there exists a polynomial $p:\rn\to\r$ such that
$$\sup_{x\in\cal{M}}\mmag{\partial_{\mathbf{2}}\tilde h(x)-p(x)}\leq \varepsilon,$$
where $\mathbf{2}:=(2,2,\dots,2)$ the tuple in $\nn$ whose entries are all $2$. Let $q(x):=G_\mathbf{2}p(x)$ and choose any $\alpha\in\nn_2$. Since $\partial_\alpha\tilde{h}$ has support in $[-l,l]^n$, applying repeatedly the fundamental theorem of calculus we have that $\partial_\alpha \tilde{h}=G_{\mathbf{2}-\alpha}\partial_{\mathbf{2}}\tilde{h}$. Applying repeatedly Leibniz integral rule we then have that
$$\partial_\alpha \tilde{h}-\partial_\alpha q = G_{\mathbf{2}-\alpha}\partial_{\mathbf{2}-\alpha}\partial_\alpha \tilde{h} -\partial_\alpha G_\mathbf{2}p=G_{\mathbf{2}-\alpha}(\partial_\mathbf{2}\tilde{h}-p)$$
Since for any continuous function $f$ and $i=1,\dots,n$, it is the case that $\mmag{G_if}\leq G_i\mmag{f}$, we have that
$$\sup_{x\in\cal{M}}\mmag{\partial_\alpha h(x)-\partial_\alpha q(x)}=\sup_{x\in\cal{M}}\mmag{\partial_\alpha \tilde{h}(x)-\partial_\alpha q(x)}\leq \sup_{x\in\cal{M}}\left(G_{\mathbf{2}-\alpha}\mmag{\partial_\mathbf{2}\tilde{h}-p}(x)\right)$$
$$\leq G_{\mathbf{2}-\alpha}\left(\sup_{x\in\cal{M}}\mmag{\partial_\mathbf{2}\tilde{h}(x)-p(x)}\right).$$
From the definition of $p$ we have that the righthand side of the above is less or equal than $(2l)^{|\mathbf{2}-\alpha|}\varepsilon$. Since the $\varepsilon$ was arbitrary, \eqref{eq:approx} follows.
\end{proof}
\begin{lemma}\label{FP} Under the conditions of Theorem \ref{convl} we have that, for every $t\geq0$ and $x\in\rn$
$$\Ebb{(\cal{A}h)(\Lambda_t^x)}=(\cal{A}u_t)(x)$$
\end{lemma}
\begin{proof}Choose any $t,s\geq0$ and $x\in\rn$. For typographical condition we write we use $\Lambda^x_t$ and $\Lambda(t,x)$ interchangeably in this proof. If $\mu_s^x$ denotes the law of $\Lambda_s^x$, by definition of $u_t$ we have that
$$\Ebb{u_t(\Lambda_s^x)}=\int\Ebb{h(\Lambda_t^y)}\mu_s^x(dy)=\Ebb{h(\Lambda(t,\Lambda(s,x))}.$$
By the Markov property of $\Lambda$ we have that 
$$\Ebb{h(\Lambda(t,\Lambda(s,x))}=\Ebb{h(\Lambda(t+s,x))}.$$
Applying It\^o's rule we have that
$$h(\Lambda(t+s,x))=h\left(\Lambda(t,x)\right)+\int_t^{t+s}(\cal{A}h)(\Lambda_v^x)dv+\int_t^{t+s}\iprod{U(\Lambda_v^Z)\nabla h(\Lambda_v^Z)}{dW_v}.$$
Since the paths of $\Lambda$ take values in a compact set the rightmost term above is a martingale (as a function of $s$). Taking expectations and applying Fubini's Theorem we have that
$$\Ebb{u_t(\Lambda_s^x)}=u_t(x)+\int_t^{t+s}\Ebb{(\cal{A}h)(\Lambda_v^x)}dv$$
Thus,
\begin{equation}\label{eq:limm}\lim_{s\to0}\frac{\Ebb{u_t(\Lambda_s^x)}-u_t(x)}{s}=\Ebb{(\cal{A}h)(\Lambda_t^x)}.\end{equation}
Next, applying It\^o's rule we have that
$$u_t(\Lambda_s^x)=u_t(x)+\int_0^s(\cal{A}u_t)(\Lambda_v^x)dv+\int_0^s\iprod{U(\Lambda_v^Z)\nabla u_t(\Lambda_v^Z)}{dW_v}.$$
Similarly, since the paths of $\Lambda$ take values in a compact set, the rightmost term is a martingale. Thus taking expectations of the above we have that
$$\Ebb{u_t(\Lambda_s^x)}-u_t(x)=\Ebb{\int_0^s(\cal{A}u_t)(\Lambda_v^x)dv}=s(\cal{A}u_t)(x)+\Ebb{\int_0^s(\cal{A}u_t)(\Lambda_v^x)-(\cal{A}u_t)(\Lambda_0^x)dv}.$$
Applying It\^o's rule, taking expectations, exploiting compactness of the paths of $\Lambda$, and applying Fubini's Theorem we have that
$$\mmag{\Ebb{\int_0^s(\cal{A}u_t)(\Lambda_v^x)-(\cal{A}u_t)(\Lambda_0^x)dv}}=\mmag{\Ebb{\int_0^s\int_0^v\cal{A}^2u_t(\Lambda_w^x)dw}}\leq s^2 \left(\max_{x\in\mathbb{S}^{n-1}}\mmag{(\cal{A}^2u_t)(x)}\right).$$
Consequently,
$$\lim_{s\to0}\frac{\Ebb{u_t(\Lambda_s^x)}-u_t(x)}{s}=\cal{A}u_t(x).$$
Comparing the above with \eqref{eq:limm} gives the desired result.
\end{proof}
\end{document}